\documentclass[12pt]{amsart}
\usepackage{amssymb}
\usepackage[all]{xy}

\usepackage{verbatim}
\usepackage{eucal}
\usepackage{latexsym}
\usepackage{amsfonts,amssymb,amsxtra}
\usepackage{relsize}
\usepackage[dvips]{color}
\usepackage{latexsym}
\usepackage{mathptmx}
\usepackage{amssymb}
\usepackage{graphicx}
\usepackage{amsthm}
\usepackage{amsfonts}
\usepackage{amscd}
\usepackage{bm}
\usepackage[utf8]{inputenc}

\usepackage{verbatim} 
\usepackage{eucal} 
\usepackage{color}

\usepackage[draft]{todonotes}
\usepackage{pgf}
\usepackage{tikz}
\usetikzlibrary{arrows,automata}

\usepackage{pgf, tikz}
\definecolor {processblue}{cmyk}{0.96,0,0,0}
\input xy
\xyoption{all}
\usepackage{mathtools}
\usepackage{amsfonts}
\usepackage{amsthm}
\usepackage{mathrsfs}

\usepackage{amscd}
  \newtheorem{The}{Theorem}[section]
  \newtheorem{Pro}[The]{Proposition}
  \newtheorem{Lem}[The]{Lemma}
  \newtheorem{Cor}[The]{Corollary}
  \newtheorem{Def}[The]{Definition}
  \newtheorem{Rem}[The]{Remark}

\newcommand{\bsm}{\begin{smallmatrix}}
\newcommand{\esm}{\end{smallmatrix}}
\newcommand{\bbm}{\begin{matrix}}
\newcommand{\ebm}{\end{matrix}}

\newtheorem*{theorem a*}{Theorem A}
\newtheorem*{theorem b*}{Theorem B}
\newtheorem*{corollary a*}{Corollary A}
\newtheorem*{theorem c*}{Theorem C}
\newtheorem*{corollary b*}{Corollary B}

\theoremstyle{definition}

\theoremstyle{plain}

\theoremstyle{definition}

\numberwithin{equation}{section}

\begin{document}

\title[Morita equivalence and Morita duality for rings with local units]{Morita equivalence and Morita duality for rings with local units and the subcategory of projective unitary modules}
\newcommand\shortTitle{Morita equivalence and Morita duality for rings with local units}
\author{Ziba Fazelpour}
\address{School of Mathematics, Institute for Research in Fundamental Sciences (IPM), P.O. Box: 19395-5746, Tehran, Iran}
\email{z.fazelpour@ipm.ir}
\author{Alireza Nasr-Isfahani}
\address{Department of Pure Mathematics\\
Faculty of Mathematics and Statistics\\
University of Isfahan\\
Isfahan 81746-73441, Iran\\ and School of Mathematics, Institute for Research in Fundamental Sciences (IPM), P.O. Box: 19395-5746, Tehran, Iran}
\email{nasr$_{-}$a@sci.ui.ac.ir / nasr@ipm.ir}



\begin{abstract}
We study Morita equivalence and Morita duality for rings with local units. We extend the Auslander's results on the theory of Morita equivalence and the Azumaya-Morita duality theorem to rings with local units. As a consequence, we give a version of Morita theorem and Azumaya-Morita duality theorem over rings with local units in terms of their full subcategory of finitely generated projective unitary modules and full subcategory of finitely generated injective unitary modules.

\end{abstract}

\maketitle

\section{Introduction}
Two rings with identity are called equivalent when their categories of modules in the category-theoretical sense are equivalent (i.e. there is a functor between their module categories, which has an inverse functor). Kiiti Morita proved a theorem which gave an equivalent condition for this equivalence in an influential paper \cite{M} and, in his honour, this equivalence of module categories between two rings has been dubbed Morita equivalence. This theorem plays an important role in modern algebra. Nowadays, this phenomena is applied in different but closely related senses in a wide range of mathematical fields (see \cite{ME} and references therein). Any property of modules defined purely in terms of modules and their homomorphisms (and not to their underlying elements or ring) is a categorical property which will be preserved by the equivalence functor. If two rings are Morita equivalent, there is an induced equivalence of the respective categories of projective modules since the Morita equivalences will preserve exact sequences and hence projective modules. In \cite[Proposition 6.5]{ausla1} Auslander proved that for each finitely generated projective left module $P$ over a ring $\Lambda$ with identity, the category ${\rm End}_{\Lambda}(P)$-Mod is equivalent to the full subcategory of $\Lambda$-Mod which its objects are all $P$-torsion $P$-generated left $\Lambda$-modules (see also \cite[Ch. XI, Sect. 8]{bs}). In fact, the Auslander's result is a generalization of the Morita theorem. Moreover he gave a version of the Morita theorem over rings with identity in terms of their full subcategory of finitely generated projective modules. More precisely he proved that two rings $\Lambda$ and $\Gamma$ with identity are Morita equivalent if and only if there exists an additive equivalence between the category of finitely generated projective $\Lambda$-modules and the category of finitely generated projective $\Gamma$-modules \cite[Proposition 2.8]{ausla1}.

In ring theory and representation theory of associative algebras, many authors make the assumption that all rings and algebras have an identity element. There are, however, lots of natural constructions in ring theory and representation theory of associative algebras which share almost all properties of rings and algebras with identity except the property of having an identity element. Such constructions include ideals, infinite direct sums of rings, category rings where the category has infinitely many objects (see e.g. \cite[Proposition 4]{L}), Leavitt path algebras of quivers with infinitely many vertices (see e.g. \cite[Lemma 1.2.12(iv)]{aam}), path algebras of quivers with infinitely many vertices (see e.g. \cite{BLP}) and linear transformations of finite rank of an infinite dimensional vector space. For many classes of rings and algebras lacking an identity element there still exist weaker versions of identity elements. One of the important class of such rings and algebras is the class of rings (algebras) with local units. A ring $R$ is called a {\it ring with local units} if every finite subset of $R$ contained in a subring of the form $eRe$ where $e^2 = e \in R$ \cite{am}. Rings with local units occur widely in mathematics, for example in algebra, functional analysis and category theory. Infinite direct sums of rings, category rings where the category has infinitely many objects, Leavitt path algebras of quivers with infinitely many vertices, path algebras of quivers with infinitely many vertices and linear transformations of finite rank of an infinite dimensional vector space are important examples of rings (algebras) with local units.

Natural classes of rings with local units appeared in functor categories. Gabriel, in \cite{ga}, proved that for any ring $\Lambda$ with identity the functor category Mod$(\Lambda$-mod$)$ is equivalent to the category of unitary left modules over a ring with local units. Wisbauer,  in \cite{wis}, gave a characterization of rings of finite representation type $\Lambda$ in terms of the rings with local units associated to Mod$(\Lambda$-mod$)$ (see also \cite{ausla2}). Fuller in \cite{f} showed that a ring $\Lambda$ with identity is left pure semisimple if and only if the ring with local units associated to Mod$(\Lambda$-mod$)$ is left perfect.

Gene D. Abrams was the first person who studied the theory of Morita equivalence for the class of rings with local units \cite{a}. Let $R$ and $S$ be two rings with local units. He proved
the category of unitary left $R$-modules and the category of unitary left $S$-modules are equivalent
if and only if there exists a unitary left $R$-module $P$ which is a generator and $S$ is isomorphic to the ring of certain endomorphisms of $P$ \cite[Theorem 4.2]{a}. In this  case $P$ is the
direct limit of a given kind of system of finitely generated projective modules. Later P. N. \'{A}nh and L. M\'{a}rki in \cite{am} extended the Abrams results and obtained the two-sided characterizations of Morita equivalence for rings with local units (see \cite[Theorems 2.1, 2.2 and 2.4]{am}).

In this paper we continue the study of Morita equivalence and Morita duality for rings with local units. Let $R$ be a ring with local units. A unitary left $R$-module $P$ is called {\it locally projective} if $P$ is a direct limit of a split direct system $(P_i)_{i \in I}$, where each $P_i$ is a finitely generated projective unitary left $R$-module \cite{am}. Let $S$ be a ring with local units and $M$ be an $R$-$S$-bimodule. We denote by $\mathscr{C}_M$ the full subcategory of $R$Mod (i.e. the category of unitary left  R-modules), whose objects are the unitary left $R$-modules $Z'$ such that the sum of all $R$-submodules $K$ of $Z'$ with $S{\rm Hom}_R(P,K)=0$ is zero and the submodule $\sum_{g \in S{\rm Hom}_R(P, Z')}{\rm Im}g$ of $Z'$ is equal to $Z'$. First, we extend the Auslander's result on Morita theory (\cite[Proposition 6.5]{ausla1}) to rings with local units which is a reformulation of the Morita theorem. More precisely we prove the following theorem.

\begin{theorem a*}$($Theorem \ref{a8}$)$
Let $R$ be a ring with local units and $P$ be a locally projective left $R$-module. Assume that $S$ is a subring of ${\rm End}_R(P)$ such that $P\in {\rm Mod}S$, $S{\rm End}_R(P)=S$ and $Pf$ is a finitely generated left $R$-module for each $f^2=f \in S$. Then the functor $S{\rm Hom}_R(P,-):\mathscr{C}_P \rightarrow S{\rm Mod}$ is an equivalence of categories.
\end{theorem a*}

Let ${\rm Proj}R$ (resp., ${\rm proj}R$) be a full subcategory of $R$Mod consisting of (resp., finitely generated) unitary left $R$-modules which are projective in $R$Mod. We prove the following theorem which gives an equivalent condition for Morita equivalence and extend \cite[Proposition 2.8]{ausla1}.

\begin{theorem b*}$($Theorem \ref{d}$)$
Let $R$ and $S$ be two rings with local units. Then the following conditions are equivalent.
\begin{itemize}
\item[$(a)$] The rings $R$ and $S$ are Morita equivalent.
\item[$(b)$] There exists an additive equivalence ${\rm Proj}R\rightarrow {\rm Proj}S$ which preserves and reflects finitely generated projective unitary modules.
\item[$(c)$] There exists an additive equivalence ${\rm proj}R \rightarrow {\rm proj}S$.
\end{itemize}
\end{theorem b*}

Dual to the Morita theory is the theory of duality between the module categories, where the functors used are contravariant rather than covariant, which is called Morita duality theory. This theory, though similar in form, has significant differences because there is no duality between the categories of modules for any rings, although dualities may exist for subcategories. P. N. \'{A}nh and C. Menini in \cite{ame} modified the concept of linear compactness and gave a characterization of Morita duality over rings with local units which is analogous to the classical case of rings with identity (see \cite[Theorem 1]{ame}).

A ring $R$ is called left (resp., right) locally finite if every finitely generated unitary left (resp., right) R-module has finite length. In this paper by using of locally finite rings and Morita duality bimodules (see Definition \ref{Def1}) we give an equivalent condition for Morita duality of rings with local units. More precisely we prove the following theorem which is a generalization of Azumaya-Morita duality theorem.

\begin{theorem c*}$($Theorem \ref{rdu3}$)$\label{C}
Let $R$ and $S$ be two rings with local units. Then the following conditions are equivalent.
\begin{itemize}
\item[$(a)$] There exists a duality between the category of finitely generated unitary left $R$-modules and the category of finitely generated unitary right $S$-modules.
\item[$(b)$] $R$ is a left locally finite ring and there exists a Morita duality $R$-$S$-bimodule.
\item[$(c)$] $S$ is a right locally finite ring and there exists a Morita duality $R$-$S$-bimodule.
\end{itemize}
In this case, if $U$ is a Morita duality $R$-$S$-bimodule, then the duality  functors are isomorphic to ${\rm Hom}_{R}(-,U)S: R{\rm mod} \rightarrow {\rm mod}S$ and $R{\rm Hom}_{S}(-,U): {\rm mod}S \rightarrow R{\rm mod}$.
\end{theorem c*}

Let $R$ be a ring with local unit. $R$ is called {\it left Morita} if there exists a ring $R'$ with local units and a duality $R{\rm mod} \rightarrow {\rm mod}R'$. Finally, as a consequence of Theorems B and C we prove following corollary.

\begin{corollary a*}$($Corollary \ref{rdu4}$)$
Let $R$ and $S$ be two rings with local units. Then the following conditions are equivalent.
\begin{itemize}
\item[$(a)$] There exists a duality between the category of finitely generated unitary left $R$-modules and the category of finitely generated unitary right $S$-modules.
\item[$(b)$] $R$ is a left Morita ring and there exists a duality ${\rm inj}R \rightarrow {\rm proj}S^{\rm op}$.
\end{itemize}
\end{corollary a*}

In \cite{FN} and forthcoming paper, the authors study some ring and module properties by using some homological conditions. This idea goes back to Auslander. Auslander in his fames theorem proved that there is a bijection between the Morita equivalence classes of non-semisimple representation finite rings and the Morita equivalence classes of non-semisimple Auslander rings, where an artinian ring $R$ is called Auslander ring if the global dimension of $R$ is less than or equal 2 and the dominant dimension of $R$ is greater than or equal 2 (see \cite[Theorem 4.6 and Corollary 4.7]{ausla2}). This result is one of the fundamental results in representation theory. In our project we prove similar bijections for some important classes of rings and by using these bijections we will try to classify some important classes of rings. The functor rings that appeared in the proof of the Auslander's bijection are unital, but in our case the functor rings are not unital, but they are rings with local units. Therefore we need the results of this paper in \cite{FN} and our forthcoming paper. We believe the results of this paper also has independent interest.

The paper is organized as follows. In Section 2, we prove some preliminary results that will be needed later in the paper. In Section 3, we first give a generalization of the Morita theorem for rings with local units (see Theorem \ref{a8}). Then we show that a ring $R$ with local units is Morita equivalent to a ring $S$ with local units if and only if there exists an additive equivalence between the full subcategory of finitely generated projective unitary left $R$-modules and the full subcategory of finitely generated projective unitary left $S$-modules (see Theorem \ref{d}). In Section 4, we prove the Azumaya-Morita duality theorem to rings with local units (see Theorem \ref{rdu3}). Moreover, by using some results of section 3, we show that there exists a duality between the category of finitely generated unitary left $R$-modules and the category of finitely generated unitary right $S$-modules if and only if $R$ is a left Morita ring and there exists a duality ${\rm inj}R \rightarrow {\rm proj}S^{\rm op}$ (see Corollary \ref{rdu4}).


\subsection{Notation }
Throughout this paper all rings are associative with local units. Let $R$ be a ring. We write all homomorphisms of left (resp., right) $R$-modules on the right (resp., left), so $fg$ (resp., $g \circ f$) means ``first $f$ then $g$". Moreover we write all ring homomorphisms on the right. We denote by $R$-Mod (resp., Mod-$R$) the category of all left (resp., right) $R$-modules.  A left $R$-module $M$ is called {\it unitary} if $RM=M$. We denote by $R$Mod (resp., Mod$R$) the category of all unitary left (resp., right) $R$-modules. Note that the category of unitary modules over a ring with local units is a Grothendieck category \cite[Page 1]{gra}. Also we denote by $R$mod (resp., mod$R$) the category of all finitely generated unitary left (resp., right) $R$-modules. Note that if a unitary module $M$ generates by a set $X$, then we denoted it by $<X>$.  We write $\bigoplus_A M$ and $\prod_A M$ for the direct sum of cardinal $A$ copies of an $R$-module $M$ and the direct product of cardinal $A$ copies of an $R$-module $M$, respectively. Let $N$ be a unitary left $R$-module. We denote by $E(N)$ the injective hull of $N$ in $R$Mod. We denote by ${\rm Proj}R$ (resp., ${\rm proj}R$) the full subcategory of $R$Mod consisting of (resp., finitely generated) unitary left $R$-modules which are projective in $R$Mod.  Also we denote by ${\rm inj}R$ the full subcategory of $R$Mod consisting of finitely generated unitary left $R$-modules which are injective in $R$Mod. Moreover we denote by ${\rm proj}R^{\rm op}$ the full subcategory of Mod$R$ consisting of finitely generated unitary right $R$-modules which are projective in Mod$R$. Let $S$ be a ring and $P$ be an $R$-$S$-bimodule. We denote by ${\rm Ker}_P$  the full subcategory of $R$Mod consisting of all unitary left $R$-modules $Z$ such that $S$Hom$_R(P,Z)=0$. Let $Z$ be a unitary left $R$-module. We denoted by ${\rm st}_P(Z)$ the sum of all $R$-submodules $K$ of $Z$ with $S{\rm Hom}_R(P,K)=0$. Moreover we recall that the submodule ${\rm sTr}(P,Z)= \sum_{g \in S{\rm Hom}_R(P, Z)}{\rm Im}g$ of $Z$ is called {\it s-trace of $P$ {\rm in} $Z$}. We denote by $\mathscr{C}_P$ the full subcategory of $R$Mod, whose objects are the unitary left $R$-modules $M$ such that ${\rm st}_P(M)=0$ and ${\rm sTr}(P,M)=M$. We recall that $P$ is called {\it unitary $R$-$S$-bimodule} if it is unitary on both sides. Let $f : X \rightarrow Y$ be an $R$-module homomorphism and assume that $M$ is an $R$-submodule of $X$. We denote by $f\mid_M$ the restriction of $f$ to $M$.

\section{Preliminaries}

Let $P$ be an $R$-$S$-bimodule and $M$ be a left $R$-module. Then by \cite[Theorem 4.8]{hu}, ${\rm Hom}_R(P,M)$ is a left $S$-module. It is easy to see that $S{\rm Hom}_R(P,M)$ is the largest unitary left $S$-submodule of ${\rm Hom}_R(P,M)$.  By $S{\rm Hom}_{R}(P,-)$ we denote the functor induced by the mapping $N \mapsto S{\rm Hom}_{R}(U, N)$ which is a left exact functor (see \cite[Page 3]{am}). The following Lemma \ref{a0} can be proved along the same lines as in the case of rings with identity (see \cite[Page 261]{and}). Therefore we present it without proof.

\begin{Lem}\label{a0}
Let $R$ and $S$ be two rings and $P$ be an $R$-$S$-bimodule which is a unitary left $R$-module. Then the pair of functors
\begin{center}
$P\otimes_S-:S{\rm Mod} \rightarrow R{\rm Mod}$ \hspace{2mm} and \hspace{2mm} $S{\rm Hom}_R(P,-):R{\rm Mod} \rightarrow S{\rm Mod}$
\end{center}
is adjoint via the isomorphisms {\rm(}for $M \in R${\rm Mod} and $N \in S${\rm Mod)}
\begin{center}
$\psi_{N,M}:{\rm Hom}_S(N,S{\rm Hom}_R(P,M)) \rightarrow {\rm Hom}_R(P\otimes_SN,M)$ via $\beta \mapsto \left[ \sum_{i=1}^tx_i\otimes n_i\mapsto \sum_{i=1}^t(x_i)(n_i)\beta \right].$
\end{center}
\end{Lem}

Let $(I, \leq)$ be a quasi-ordered directed set. A direct system of unitary left $R$-modules\\ ${(X_i, \varphi_{ij})_{i,j \in I}}$ is called {\it split} if for each $i \leq j$ in $I$, there exists an $R$-module homomorphism $\psi_{ji}: X_j \rightarrow X_i$ such that $\varphi_{ij}\psi_{ji}=id_{X_i}$ and $\psi_{kj}\psi_{ji}=\psi_{ki}$ for each $i \leq j \leq k$ in $I$ (see \cite{v}). We denoted it by $\lbrace X_i, \varphi_{ij}, \psi_{ji}~|~ i,j \in I \rbrace$.

\begin{Lem}\label{a1}
Let $R$ be a ring and let $\lbrace P_i, \varphi_{ij}, \psi_{ji}~|~ i,j \in I \rbrace$ be a split direct system of finitely generated unitary left $R$-modules. Assume that $\lbrace \varphi_i: P_i \rightarrow P~|~ i \in I\rbrace$ is a direct system of $R$-module homomorphisms from $(P_i, \varphi_{ij})_{I}$ and let $P$ together with the $R$-module homomorphisms $\varphi_i$ form a direct limit of $(P_i, \varphi_{ij})_{I}$. Then
\begin{itemize}
\item[$(a)$] there exists a unique $R$-module homomorphism $\psi_i:P \rightarrow P_i$ such that $\varphi_i\psi_i=id_{P_i}$ and $\psi_j\psi_{ji}=\psi_i$ for each $i \leq j$ in $I$.
\item[$(b)$] $S = {\underrightarrow{\lim}}{\rm End}_R(P_i)= \lbrace g \in {\rm End}_R(P) ~|~ g ~{\rm factors ~through ~one ~of ~the ~projections} ~\psi_i \rbrace$ is a subring of ${\rm End}_R(P)$ such that  $S{\rm End}_R(P)=S$, $P$ is a unitary $R$-$S$-bimodule and $Pf$ is a finitely generated unitary left $R$-module for each $f^2=f \in S$.
\end{itemize}
\end{Lem}
\begin{proof} $(a)$. It follows from \cite[Proposition 1.1]{v}.\\
$(b)$. It follows from \cite[Page 11]{am}.
\end{proof}

We recall from \cite{am} that a unitary left $R$-module $P$ is called {\it locally projective} if $P$ is a direct limit of a split direct system $(P_i)_{i \in I}$, where each $P_i$ is a finitely generated  projective unitary left $R$-modules.  Note that  $_RR$ is a locally projective module since it is a direct limit of a split direct system $\lbrace Re, \alpha_{ee'}, \rho_{e'e}~|~e,e' \in E\rbrace$, where $E$ be the set of all idempotents of $R$ and each $\alpha_{ee'}:Re\rightarrow Re'$ and $\rho_{e'e}:Re' \rightarrow Re$ are the multiplication maps. Also $R$ together the canonical ring homomorphisms $t_e: eRe \rightarrow R$ form a direct limit of a direct system $(eRe, t_{ee'})_{e,e' \in E}$, where each $t_{ee'}:eRe \rightarrow e'Re'$ is the canonical injection (see \cite{am}).

\begin{Rem}{\rm
Let $R$ and $S$ be two rings with local units and $P$ be a unitary $R$-$S$-bimodule. Assume that $P$ is a locally projective unitary left $R$-module and   $Pf$ is a finitely generated left $R$-module, where $f$ is an idempotent of $S$. Since $P$ is a locally projective module, $P$ is a direct limit of a split direct system $\lbrace P_i, \varphi_{ij}, \psi_{ji}~|~ i,j \in I \rbrace$, where each $P_i$ is a finitely generated  projective unitary left $R$-module.  Assume that $\lbrace \varphi_i: P_i \rightarrow P~|~ i \in I\rbrace$ is a direct system of $R$-module homomorphisms from $(P_i, \varphi_{ij})_{I}$ and let $P$ together with the $R$-module homomorphisms $\varphi_i$ form a direct limit of $(P_i, \varphi_{ij})_{I}$. Then by Lemma 2.2, there exists an $R$-module homomorphism $\psi_i:P \rightarrow P_i$ such that $\varphi_i\psi_i=id_{P_i}$ for each $i \in I$. This implies that each $\varphi_i$ is a split monomorphism. Since $Pf$ is a finitely generated submodule of $P$, by \cite[Proposition 24.3(3)]{wi}, $Pf \subseteq {\rm Im}\varphi_j$ for some $j \in I$.  Since $\varphi_j$ is a split monomorphism, ${\rm Im}\varphi_j$ is a finitely generated projective unitary direct summand of $P$. Set $P^{'}= {\rm Im}\varphi_j$. It is easy to see that $P^{'}= Pf \oplus P^{'}(1-f)$, where $P^{'}(1-f)= \lbrace x \in P^{'}~|~ xf=0 \rbrace$. This yields that $Pf$ is a finitely generated projective unitary direct summand of $P$.}
\end{Rem}

\begin{Lem}\label{a4}
Let $R$ and $S$ be two rings and $P$ be an $R$-$S$-bimodule. If $P$ is a locally projective left $R$-module and $Pf$ is a finitely generated left $R$-module for each $f^2=f \in S$, then the functor $S{\rm Hom}_R(P,-):R{\rm Mod} \rightarrow S{\rm Mod}$ is exact.
\end{Lem}
\begin{proof}
Let $P$ be a locally projective left $R$-module and assume that $Pf$ is a finitely generated left $R$-module for each $f^2=f \in S$. We show that the functor $S{\rm Hom}_R(P,-):R{\rm Mod} \rightarrow S{\rm Mod}$ is exact. It is enough to show that the functor $S{\rm Hom}_R(P,-):R{\rm Mod} \rightarrow S{\rm Mod}$ preserves epimorphisms. Let $g: N \rightarrow L$ be an epimorphism in $R$Mod. Assume that $\gamma \in S{\rm Hom}_R(P,L)$. Then there exists an idempotent $f \in S$ such that $f \gamma=\gamma$. Since $P$ is a locally projective left $R$-module and $Pf$ is a finitely generated left $R$-module, $Pf$ is a finitely generated projective direct summand of $P$. Consequently there exists an $R$-module homomorphism $h':Pf \rightarrow N$ such that $h'g=\epsilon \gamma$, where $\epsilon:Pf \rightarrow P$ is the canonical injection. Let $\pi:P \rightarrow Pf$ be the canonical projection and $x \in P$. Since $(x)\pi h'=(xf)h'=(xf)\pi h'=(x)f\pi h'$, $\pi h' \in S{\rm Hom}_R(P,N)$. Since $(x)\pi h'g=(x)\pi \epsilon \gamma = (xf)\gamma= (x)f\gamma=(x)\gamma$, $\pi h'g=\gamma$. Therefore $S{\rm Hom}_R(P,g)$ is an epimorphism.
\end{proof}

\begin{Lem}\label{a2}
Let $R$ and $S$ be two rings and $P$ be an $R$-$S$-bimodule. If  $P$ is a locally projective left $R$-module and $Pf$ is a finitely generated left $R$-module for each $f^2=f \in S$, then the functor $R{\rm Hom}_S(G_P,-): S{\rm Mod}\rightarrow R{\rm Mod}$ is a right adjoint of the functor $S{\rm Hom}_R(P,-): R{\rm Mod} \rightarrow S{\rm Mod}$, where $G_P=S{\rm Hom}_R(P,R)$.
\end{Lem}
\begin{proof}
Let $M$ be a unitary left $R$-module and set $G_P=S{\rm Hom}_R(P,R)$. It is easy to see that $G_P$ is a $S$-$R$-bimodule which is a unitary left $S$-module. Let $m \in M$ and $\gamma \in G_P$. Then there exists $f^2=f \in S$ such that $f\gamma = \gamma$. Since $(xf)\gamma m = (x)(f\gamma)m = (x)\gamma m$ for each $x \in P$, it is easy to check that the mapping $G_P \times M \rightarrow S{\rm Hom}_R(P,M)$ defined by $(\gamma,m) \longmapsto (\gamma \otimes m)\eta_M:P \rightarrow M$ via $x \longmapsto (x)\gamma m$ is $R$-balanced.  It follows that there exists a unique homomorphism of abelian groups $\eta_M: G_P \otimes_R M  \rightarrow S{\rm Hom}_R(P,M)$ defined by $\sum_{i=1}^t\gamma_i\otimes m_i \longmapsto \sum_{i=1}^t(\gamma_i\otimes m_i)\eta_M$. It is easy to see that  $\eta_M$ is an $S$-module homomorphism. Assume that $\sum_{i=1}^t\gamma_i\otimes m_i \in G_P \otimes_R M$ such that $(\sum_{i=1}^t\gamma_i\otimes m_i)\eta_M = 0$. Then $\sum_{i=1}^t (x)\gamma_i m_i = 0$ for each $x \in P$. We know that there exists an idempotent $f^2=f \in S$ such that $f \gamma_i = \gamma_i$ for each $1 \leq i \leq t$. Hence for each $1 \leq i \leq t$, $\gamma_i=\pi \varepsilon \gamma_i$, where $\pi: P \rightarrow Pf$ is the canonical projection and $\varepsilon: Pf \rightarrow P$ is the canonical injection. Set ${\gamma}^{'}_i=\varepsilon \gamma_i$. Since $Pf$ is a finitely generated projective left $R$-module, it is not difficult to see that by \cite[Propositions 1.2 and 1.5]{am}, $\sum_{i=1}^t{\gamma}^{'}_i\otimes m_i =0$ in ${\rm Hom}_R(Pf,R) \otimes_R M$. Since the mapping ${\rm Hom}_R(Pf,R) \rightarrow f{\rm Hom}_R(P,R)$ via $\alpha \mapsto f \pi \alpha$ is an $R$-module isomorphism, $\eta_M$ is a monomorphism. Let $\phi \in S{\rm Hom}_R(P,M)$. Then there exists an idempotent $f \in S$ such that $f\phi = \phi$. Hence $\phi\mid_{P(1-f)} = 0$. Since $Pf$ is a finitely generated projective left $R$-module, it follows that by \cite[Propositions 1.2 and 1.5]{am}, there exists $\sum_{i=1}^t \alpha'_i \otimes m_i \in {\rm Hom}_R(Pf,R)\otimes_R M$ such that $\sum_{i=1}^t(xf){\alpha'_i}m_i = (xf)\phi$ for each $x \in P$.  We define the mapping $\alpha_i:P \rightarrow R$ via $x \mapsto (xf)\alpha'_i$. Clearly each $\alpha_i$ is well-defined and $\alpha_i \in G_P$. Since for each $x \in P$ $$\sum_{i=1}^t(x)\alpha_im_i = \sum_{i=1}^t(xf)\alpha'_im_i = (xf)\phi = (x)(f\phi) = (x)\phi,$$  $\eta_M$ is surjective. Consequently $\eta_M$ is an $S$-module isomorphism. Therefore for each unitary left $S$-module $N$, ${\rm Hom}_S(\eta_M,N):{\rm Hom}_S(S{\rm Hom}_R(P,M),N) \rightarrow {\rm Hom}_S(G_P\otimes_RM,N)$ is an isomorphism of abelian groups. Let $h:M \rightarrow M'$ be an $R$-module homomorphism. Since for each $\sum_{i=1}^t\beta_i\otimes m_i \in G_P \otimes_RM$ and $x \in P$ $$(x)(\sum_{i=1}^t\beta_i\otimes (m_i)h)\eta_{M'}=\sum_{i=1}^t(x)\beta_i(m_i)h=(\sum_{i=1}^t(x)\beta_im_i)h=(x)(\sum_{i=1}^t\beta_i\otimes m_i)\eta_Mh,$$
the following diagram is commutative
\begin{displaymath}
\xymatrixcolsep{5pc}\xymatrix{
{\rm Hom}_S(S{\rm Hom}_R(P,M),N) \ar[r]^<<<<<<<{{\rm Hom}_S(\eta_{M},N)}  &
{\rm Hom}_S(G_P\otimes_R M,N)  \\
{\rm Hom}_S(S{\rm Hom}_R(P,M'),N) \ar[u]^{{\rm Hom}_S(S{\rm Hom}_R(P,h),N)}  \ar[r]_<<<<<<<{{\rm Hom}_S(\eta_{M'},N)} & {\rm Hom}_S(G_P\otimes_R M',N).\ar[u]_{{\rm Hom}_S(1\otimes h,N)}}
\end{displaymath}
Therefore by Lemma \ref{a0}, the functor $R$Hom$_S(G_P,-)$ is a right adjoint of the functor\\ $S$Hom$_R(P,-)$.
\end{proof}

In the following Lemmas \ref{a3}-\ref{a5}, \ref{a6} and Corollary \ref{a50},  we assume that $P$ is a locally projective left $R$-module, $S$ is a subring of ${\rm End}_R(P)$ such that $P$ is a unitary right $S$-module, $S{\rm End}_R(P)=S$ and $Pf$ is a finitely generated left $R$-module for each $f^2=f \in S$.

\begin{Lem}\label{a3}
There is a natural transformation $\gamma: id_{R{\rm Mod}} \rightarrow R{\rm Hom}_S(G_P, S{\rm Hom}_R(P,-))$ such that for each unitary left $R$-module $M$, the morphism $$S{\rm Hom}_R(P,\gamma_M): S{\rm Hom}_R(P, M) \rightarrow S{\rm Hom}_R(P, R{\rm Hom}_S(G_P, S{\rm Hom}_R(P,M)))$$
is an $S$-module isomorphism which determines a natural isomorphism $$S{\rm Hom}_R(P, -) \simeq S{\rm Hom}_R(P, R{\rm Hom}_S(G_P, S{\rm Hom}_R(P,-))).$$
\end{Lem}
\begin{proof}
It follows from Lemmas \ref{a0}, \ref{a2} and \cite[Proposition 1.1]{am}.
\end{proof}

\begin{Lem}\label{b1}
For each $N \in S{\rm Mod}$, there exists an $S$-module isomorphism $$\nu_N:N \rightarrow S{\rm Hom}_R(P,R{\rm Hom}_S(G_P,N))$$
which determines a natural isomorphism ${\rm id}_{S{\rm Mod}} \rightarrow S{\rm Hom}_R(P,R{\rm Hom}_S(G_P,-))$.
\end{Lem}
\begin{proof}
It follows from Lemma \ref{a3} and \cite[Proposition 1.2 and Corollary 1.6]{am}.
\end{proof}

\begin{Lem}\label{a5}
Let $N$ be a unitary left $R$-module such that ${\rm st}_P(N) = 0$. Then
 \begin{itemize}
 \item[$(a)$] For each $X \in {\rm Ker}_P$, ${\rm Hom}_R(X,N)=0$.
 \item[$(b)$] The $R$-module homomorphism $\gamma_N: N \rightarrow R{\rm Hom}_S(G_P, S{\rm Hom}_R(P,N))$ is an isomorphism when $N$ is an injective module in $R{\rm Mod}$.
 \end{itemize}
\end{Lem}
\begin{proof}
$(a)$. Let $N$ be a unitary left $R$-module such that ${\rm st}_P(N) = 0$ and $X \in {\rm Ker}_P$.  On the contrary, assume that $g:X \rightarrow N$ is a non-zero $R$-module homomorphism. Without loss of generality we can assume $g$ is an epimorphism. Since ${\rm st}_P(N) = 0$, there exists a non-zero morphism $\alpha\in S{\rm Hom}_R(P,N)$.  Therefore there exists an idempotent $f \in S$ such that $f\alpha=\alpha$.  Since $Pf$ is a projective direct summand of $P$, there exists a non-zero $R$-module homomorphism $h:Pf \rightarrow X$ such that $hg=\epsilon \alpha$, where $\epsilon:Pf \rightarrow P$ is the canonical injection.  Let $\pi:P \rightarrow Pf$ be the canonical projection and $x \in P$. Since $(x)\pi h=(xf)h=(xf)\pi h=(x)f\pi h$, $\pi h \in S{\rm Hom}_R(P,X)$. Since $X \in {\rm Ker}_P$, $\pi h=0$. It follows that $h=0$ which is a contradiction. Therefore Hom$_R(X,N)=0$ for each $X \in {\rm Ker}_P$.\\
$(b)$. Let $N$ be an injective module in $R$Mod. Consider the exact sequence $$0 \longrightarrow K_1 \longrightarrow N \overset{\gamma_N}{\longrightarrow} R{\rm Hom}_S(G_P,S{\rm Hom}_R(P,N)) \longrightarrow K_2 \longrightarrow 0,$$ where $K_1$ is the kernel of $\gamma_N$ and $K_2$ is the cokernel of $\gamma_N$. By Lemma \ref{a4}, the functor\\ $S{\rm Hom}_R(P,-):R{\rm Mod} \rightarrow S{\rm Mod}$ is exact.  Because by Lemma \ref{a3} $S{\rm Hom}_R(P, \gamma_N)$ is an isomorphism, $S{\rm Hom}_R(P,K_1)=S{\rm Hom}_R(P,K_2)=0$. It follows that $K_1 , K_2 \in {\rm Ker}_P$. Since $K_1 \in {\rm Ker}_P$, by $(a)$, $K_1=0$. Since $N$ is an injective module in $R$Mod, the exact sequence $$0 \rightarrow N \rightarrow R{\rm Hom}_S(G_P,S{\rm Hom}_R(P,N)) \rightarrow K_2 \rightarrow 0$$ is split. Hence there exists an $R$-module monomorphism $K_2 \rightarrow R{\rm Hom}_S(G_P,S{\rm Hom}_R(P,N))$. Since by Lemma \ref{a2}, $${\rm Hom}_R(K_2,R{\rm Hom}_S(G_P,S{\rm Hom}_R(P,N)))\cong {\rm Hom}_S(S{\rm Hom}_R(P,K_2), S{\rm Hom}_R(P,N)),$$ $K_2=0$. Therefore $\gamma_N: N \rightarrow R{\rm Hom}_S(G_P, S{\rm Hom}_R(P,N))$ is an $R$-module isomorphism.
\end{proof}

\begin{Cor}\label{a50}
Let $N$ be an injective module in $R{\rm Mod}$ such that ${\rm st}_p(N)=0$. Then $S{\rm Hom}_R(P,N)$ is an injective module in $S{\rm Mod}$.
\end{Cor}
\begin{proof}
 It follows from Lemmas \ref{a4}, \ref{b1} and \ref{a5}.
\end{proof}

\begin{Lem}\label{a6}
Let $M$ be a unitary left $R$-module and $N$ be an injective module in $R{\rm Mod}$ such that ${\rm st}_P(N)=0$. Then the mapping ${\rm Hom}_R(M,N) \rightarrow {\rm Hom}_S(S{\rm Hom}_R(P,M),S{\rm Hom}_R(P,N))$ via $g \mapsto S{\rm Hom}_R(P,g)$ is an isomorphism of abelian groups.
\end{Lem}
\begin{proof}
By Lemma \ref{a5}, ${\rm Hom}_R(M, \gamma_N): {\rm Hom}_R(M,N) \rightarrow {\rm Hom}_R(M, R{\rm Hom}_S(G_P,S{\rm Hom}_R(P,N)))$ is an isomorphism of abelian groups. Also by Lemma \ref{a0}, $$\psi_{M,S{\rm Hom}_R(P,N)}: {\rm Hom}_R(M,R{\rm Hom}_S(G_P,S{\rm Hom}_R(P,N))) \rightarrow {\rm Hom}_S(G_P\otimes_RM,S{\rm Hom}_R(P,N))$$ is an isomorphism of abelian groups. By the proof of Lemma \ref{a2}, the mapping $${\rm Hom}_S(\eta^{-1}_M,S{\rm Hom}_R(P,N)): {\rm Hom}_S(G_P\otimes_RM,S{\rm Hom}_R(P,N)) \rightarrow {\rm Hom}_S(S{\rm Hom}_R(P,M),S{\rm Hom}_R(P,N))$$ is an isomorphism of abelian groups. Therefore it is enough to show that for each $g \in {\rm Hom}_R(M,N)$, $$\eta^{-1}_M(g\gamma_N)\psi_{M,S{\rm Hom}_R(P,N)}=S{\rm Hom}_R(P,g).$$  Let $\alpha \in S{\rm Hom}_R(P,M)$, $g \in {\rm Hom}_R(M,N)$ and $x \in P$. Then there exists $\sum_{i=1}^n \beta_i \otimes m_i \in G_P\otimes_R M$ such that $\alpha=(\sum_{i=1}^n \beta_i \otimes m_i)\eta_M$. Hence
\begin{center}
$(\alpha)\eta^{-1}_M(g\gamma_N)\psi_{M,S{\rm Hom}_R(P,N)}=(\sum_{i=1}^n \beta_i \otimes m_i)(g\gamma_N)\psi_{M,S{\rm Hom}_R(P,N)}=\sum_{i=1}^n (\beta_i)((m_i)g)\gamma_N=\sum_{i=1}^n(\beta_i \otimes (m_i)g)\eta_N$.
\end{center}
Since $(x)\sum_{i=1}^n(\beta_i \otimes (m_i)g)\eta_N= \sum_{i=1}^n (x)\beta_i(m_i)g = (x)\alpha g$,  $S{\rm Hom}_R(P,g)=\eta^{-1}_M(g\gamma_N)\psi_{M,S{\rm Hom}_R(P,N)}$. Therefore ${\rm Hom}_R(M,N) \cong {\rm Hom}_S(S{\rm Hom}_R(P,M),S{\rm Hom}_R(P,N))$ as abelian groups.
\end{proof}

Let $\mathcal{U}$ be a non-empty set of unitary left $R$-modules. A unitary left $R$-module $M$ is said to be {\it cogenerated by $\mathcal{U}$} if for every pair of distinct morphisms $f,g:B \rightarrow M$ in $R$Mod, there exists a morphism $h: M \rightarrow U$ with $U \in \mathcal{U}$ and $fh \neq gh$.  $\mathcal{U}$ is called a {\it cogenerating set} for $R$Mod if every unitary left $R$-module cogenerated by $\mathcal{U}$. A unitary left $R$-module $U$ is called {\it cogenerator} in $R$Mod if the set $\mathcal{U}=\lbrace U \rbrace$ is a cogenerating set for $R$Mod (see \cite[Definition 14.1]{wi}).\\

Let $R$ and $S$ be two rings and $U$ be a unitary $R$-$S$-bimodule. Assume that $M$ is a unitary left $R$-module, $K$ be an $R$-submodule of $M$ and $X$ be a $S$-submodule of ${\rm Hom}_R(M,U)S$. Set $r_{{\rm Hom}_R(M,U)S}(K)= \lbrace g \in {\rm Hom}_R(M,U)S~|~ (x)g=0 {\rm ~for ~each~} x \in K \rbrace$ and
$l_M(X)= \lbrace m \in M~|~ (m)g=0 {\rm ~for ~each}~ g \in X \rbrace$ (see \cite[Definition 28.1]{wi}).

\begin{Lem}\label{4.1}
Let $R$ and $S$ be two rings, $U$ be a unitary $R$-$S$-bimodule and $M$ and $N$ be modules in $R{\rm Mod}$ and ${\rm Mod}S$, respectively.
\begin{itemize}
\item[$(a)$] If $\lbrace Uf~|~f^2=f \in S \rbrace$ is a cogenerating set for $R{\rm Mod}$, then $l_M(r_{{\rm Hom}_R(M,U)S}(K))=K$ for each $R$-submodule $K$ of $M$.
\item[$(b)$] If $\lbrace eU~|~e^2=e \in R \rbrace$ is a cogenerating set for ${\rm Mod}S$, then $l_N(r_{R{\rm Hom}_S(N,U)}(K'))=K'$ for each $S$-submodule $K'$ of $N$.
\end{itemize}
\end{Lem}
\begin{proof}
$(a)$. Let $K$ be an $R$-submodule of $M$ and $\lbrace Uf~|~f^2=f \in S \rbrace$ be a cogenerating set for $R{\rm Mod}$. Then $l_M(r_{{\rm Hom}_R(M,U)S}(K))= \lbrace m \in M~|~(m)g=0 {\rm ~for~ each}~ g \in r_{{\rm Hom}_R(M,U)S}(K) \rbrace= \bigcap_{g \in r_{{\rm Hom}_R(M,U)S}(K)}{\rm Ker}g \supseteq K$. Assume $x \in \bigcap_{g \in r_{{\rm Hom}_R(M,U)S}(K)}{\rm Ker}g $ that $x \notin K$. There exists an idempotent $e \in R$ such that $ex=x$ and hence $Rx \neq 0$.  The mapping $\ell: Rx \rightarrow M/K$ via $rx \mapsto rx+K$ is a non-zero $R$-module homomorphism. Since $\lbrace Uf~|~f^2=f \in S \rbrace$ is a cogenerating set for $R{\rm Mod}$, there exists an $R$-module homomorphism $h: M/K \rightarrow Uf$ such that $\ell h \neq 0$, where $f$ is an idempotent of $S$. Then there exists $r \in R$ such that $(rx + K)h \neq 0$ and hence $(x + K)h \neq 0$. Let $\varepsilon: Uf \rightarrow U$ be the canonical injection. Define the mapping $\overline{h}: M \rightarrow U$ via $y \mapsto (y+K)h\varepsilon$. It is not difficult to see that $\overline{h}$ is an $R$-module homomorphism,  $(x)\overline{h} \neq 0$, $\overline{h}f=\overline{h}$ and $K \subseteq {\rm Ker}\overline{h}$. Hence $\overline{h} \in r_{{\rm Hom}_R(M,U)S}(K)$ that $(x)\overline{h} \neq 0$, which is a contradiction. Therefore $x \in K$ and hence $l_M(r_{{\rm Hom}_R(M,U)S}(K))=K$.\\
$(b)$. It is similar to the proof of the part $(a)$.
\end{proof}

A ring $R$ is called {\it left {\rm (resp.,} right{\rm )} locally noetherian} if every finitely generated unitary left (resp., right) $R$-module is noetherian (see \cite{wis}).

\begin{Lem}\label{4.2}
Let $R$ be a ring and $U$ and $M$ be unitary left $R$-modules. Assume that $S$ is a subring of ${\rm End}_R(U)$ such that it is a right locally noetherian ring, $S{\rm End}_R(U)=S$, $U$ is a unitary right $S$-module and $Uf$ is an injective module in $R{\rm Mod}$ for each idempotent $f \in S$. If ${\rm Hom}_R(M,U)S$ is a finitely generated right $S$-module, then $r_{{\rm Hom}_R(M,U)S}(l_M(X)) = X$ for each $S$-submodule $X$ of ${\rm Hom}_R(M,U)S$.
\end{Lem}
\begin{proof}
We show that $r_{{\rm Hom}_R(M,U)S}(l_M(X)) \subseteq X$. Since ${\rm Hom}_R(M,U)S$ is a finitely generated right $S$-module and $S$ is right locally noetherian, $X$ is a finitely generated right $S$-module. Let $X=<g_1,\ldots,g_k>$, where $k \in {\Bbb{N}}$ and each $g_i \in X$.  Then $l_M(X)=\lbrace m \in M~|~(m)g=0~{\rm for~each~}g \in X \rbrace= \bigcap_{g \in X}{\rm Ker}g=\bigcap_{i=1}^k{\rm Ker}g_i$.  Let $g \in r_{{\rm Hom}_R(M,U)S}(l_M(X))$. Then $g \in {\rm Hom}_R(M,U)S$ and $\bigcap_{i=1}^k{\rm Ker}g_i \subseteq {\rm Ker}g$.  There exists an idempotent $f\in S$ such that $g_if=g_i$ and $gf=g$.  We define the mapping $\overline{g}: M/\bigcap_{i=1}^k{\rm Ker}g_i \rightarrow Uf$ via $x+ \bigcap_{i=1}^k{\rm Ker}g_i \mapsto (x)g$. It is easy to see that $\overline{g}$ is an $R$-module homomorphism. We also define the mapping $\overline{g_i}: M \rightarrow Uf$ via $x \mapsto (x)g_i$.  It is not difficult to see that each $\overline{g_i}$ is an $R$-module homomorphism. Consider the following diagram
\begin{displaymath}
\xymatrix{
0 \ar[r] & \dfrac{M}{\bigcap_{i=1}^k{\rm Ker}g_i}  \ar[d]_{\overline{g}} \ar[r]^<<<<{(\overline{g_1},\ldots, \overline{g_k})}&
Uf \oplus \ldots \oplus Uf\\
& Uf, & }
\end{displaymath}
where its row is exact. Since $Uf$ is an injective module in $R$Mod, there exists an $R$-module homomorphism $h: Uf \oplus \ldots \oplus Uf \rightarrow Uf$ such that the following diagram is commutative
\begin{displaymath}
\xymatrix{
0 \ar[r] & \dfrac{M}{\bigcap_{i=1}^k{\rm Ker}g_i}  \ar[d]_{\overline{g}} \ar[r]^<<<<{(\overline{g_1},\ldots, \overline{g_k})}&
Uf \oplus \ldots \oplus Uf \ar[dl]^{h}\\
& Uf. & }
\end{displaymath}
Let $x \in M$. Then $$(x)g=(((x)g_1,\ldots,(x)g_k))h\ell=(\sum_{i=1}^k((x)g_i)\varepsilon_i)h\ell=\sum_{i=1}^k((x)g_i)\varepsilon_ih\ell=\sum_{i=1}^k((x)g_if)\varepsilon_ih\ell,$$
where each $\varepsilon_i: Uf \rightarrow Uf \oplus \ldots \oplus Uf$ is the canonical injection and $\ell: Uf \rightarrow U$ is the canonical injection. For each $1 \leq i \leq k$, set $s_i=\pi\varepsilon_ih\ell$, where $\pi: U \rightarrow Uf$ is the canonical projection. It is not difficult to see that for each $i$, $fs_i=s_i$ and so each $s_i \in S$. Since each $1 \leq i\leq k$, $(x)g_is_i=((x)g_i)s_i=((x)g_if)\varepsilon_ih\ell$, $(x)g=\sum_{i=1}^k(x)g_is_i$. It follows that $g=\sum_{i=1}^kg_is_i$ and hence $g \in X$. Therefore $r_{{\rm Hom}_R(M,U)S}(l_M(X)) = X$.
\end{proof}

\section{A Generalization of Morita's theorem for rings with local units}

In this section, we extend the Auslander's results on the theory of Morita equivalence to rings with local units. We give a reformulation of the Morita theorem for rings with local units (see Theorem \ref{a8}). Moreover we give a version of the Morita theorem over rings with local units in terms of their  full subcategory of finitely generated projective unitary modules (see Theorem \ref{d}).

\begin{Lem}\label{a7}
Let $R$ and $S$ be two rings, $P$ be an $R$-$S$-bimodule and $M$ be a unitary left $R$-module. Then
\begin{itemize}
\item[$(a)$] There exists an $R$-module epimorphism $\bigoplus_{\alpha \in S{\rm Hom}_R(P,M)} Pe_{\alpha} \rightarrow {\rm sTr}(P,M)$, where $e_{\alpha}$ is an idempotent of $S$ that $e_{\alpha}\alpha=\alpha$ for each $\alpha \in S{\rm Hom}_R(P,M)$.
\item[$(b)$] $S{\rm Hom}_R(P, M) \cong S{\rm Hom}_R(P, {\rm sTr}(P,M))$ as $S$-modules when $P$ is a locally projective left $R$-module such that $Pf$ is a finitely generated left $R$-module for each $f^2=f \in S$.
\end{itemize}
\end{Lem}
\begin{proof}
$(a)$. It is clear.\\
$(b)$. Let $P$ be a locally projective left $R$-module and assume that $Pf$ is a finitely generated left $R$-module for each $f^2=f \in S$. We show that $S{\rm Hom}_R(P, M) \cong S{\rm Hom}_R(P, {\rm sTr}(P,M))$ as $S$-modules. By Lemma \ref{a4}, it is enough to show that $S{\rm Hom}_R(P, M/{\rm sTr}(P,M)) = 0$. Suppose, contrary to our claim, that $S{\rm Hom}_R(P, M/{\rm sTr}(P,M)) \neq 0$. Let $0 \neq \alpha \in S{\rm Hom}_R(P, M/{\rm sTr}(P,M))$.  Since $\alpha \in S{\rm Hom}_R(P, M/{\rm sTr}(P,M))$,  there exists an idempotent $f \in S$ such that $f\alpha=\alpha$. Suppose that  $\epsilon:Pf \rightarrow P$ is the canonical injection and $\pi':M \rightarrow M/{\rm sTr}(P,M)$ is the canonical projection. $\alpha \neq 0$ yields $h\pi' \neq \epsilon \alpha$ for each $R$-module homomorphism $h: Pf \rightarrow M$. Indeed, if $h: Pf \rightarrow M$ is an $R$-module homomorphism such that  $h\pi' = \epsilon \alpha$, then $\pi h\pi'= \pi \epsilon \alpha$, where  $\pi:P \rightarrow Pf$ is a canonical projection. Since for each $x \in P$, $(x)\pi\epsilon \alpha= (xf)\epsilon \alpha=(xf)\alpha=(x)f\alpha=(x)\alpha$ we observe that $\pi h\pi'= \alpha$. Because $(x)\pi h=(xf)h=(xf)\pi h=(x)f\pi h$, $\pi h \in S{\rm Hom}_R(P,M)$. Consequently there is a $\beta \in S{\rm Hom}_R(P,M)$ such that $\beta \pi'=\alpha$ which is impossible. So $h\pi' \neq \epsilon \alpha$ for each $R$-module homomorphism $h: Pf \rightarrow M$. It follows that $Pf$ is not a projective left $R$-module. We conclude that either $P$ is not a locally projective left $R$-module or $Pf$ is not a finitely generated left $R$-module which is a contradiction. Therefore $S{\rm Hom}_R(P, M/{\rm sTr}(P,M)) = 0$.
\end{proof}

The following Proposition is closely related to the perfect localization induced by finitely generated projective module discussed detailedly in Section 8 of Chapter XI of \cite{bs} (see also \cite[Proposition 6.5]{ausla1}). We show this result for rings with local units. In fact we give a reformulation of the  Auslander result for rings with local units.

\begin{Pro}\label{a8}
Let $R$ be a ring and $P$ be a locally projective left $R$-module. Assume that $S$ is a subring of ${\rm End}_R(P)$ such that $P\in {\rm Mod}S$, $S{\rm End}_R(P)=S$ and $Pf$ is a finitely generated left $R$-module for each $f^2=f \in S$. Then the functor $S{\rm Hom}_R(P,-):\mathscr{C}_P \rightarrow S{\rm Mod}$ is an equivalence of categories.
\end{Pro}
\begin{proof}
We are going to show that the functor $S{\rm Hom}_R(P,-):\mathscr{C}_P \rightarrow S{\rm Mod}$ is full, faithful and dense. First we show that the functor $S{\rm Hom}_R(P,-)$ is dense.  Let $N$ be a unitary left $S$-module. From \cite[Proposition 1.2 and Corollary 1.6]{am} we get $S{\rm Hom}_R(P,P\otimes_SN)\cong S{\rm End}_R(P)S \otimes_SN = S\otimes_SN \cong N$ as $S$-modules.  Set $M:=P\otimes_SN$. Consider the exact sequence
$0 \rightarrow {\rm st}_P(M) \rightarrow M \rightarrow M/{{\rm st}_P(M)}\rightarrow 0$. We deduce from Lemma \ref{a4} that the sequence $$0 \rightarrow S{\rm Hom}_R(P,{\rm st}_P(M)) \rightarrow S{\rm Hom}_R(P,M) \rightarrow S{\rm Hom}_R(P,M/{\rm st}_P(M)) \rightarrow 0$$ is exact. From $S{\rm Hom}_R(P,{\rm st}_P(M))=0$ we get $S{\rm Hom}_R(P,M) \cong S{\rm Hom}_R(P,M/{\rm st}_P(M))$ as $S$-modules. Hence $N \cong S{\rm Hom}_R(P,{\rm sTr}(P,M/{\rm st}_P(M)))$ as $S$-modules by using of Lemma \ref{a7}(b). It is enough to show that ${\rm sTr}(P,M/{\rm st}_P(M)) \in \mathscr{C}_P$. Since for each $g \in S{\rm Hom}_R(P,M/{\rm st}_P(M))$, ${\rm Im}g \subseteq {\rm sTr}(P,M/{\rm st}_P(M))$, so  ${\rm sTr}(P,{\rm sTr}(P,M/{\rm st}_P(M)))= {\rm sTr}(P,M/{\rm st}_P(M))$. It is sufficient to show that ${\rm st}_P({\rm sTr}(P,M/{\rm st}_P(M)))=0$. Let $K$ be a non-zero $R$-submodule of $M/{\rm st}_P(M)$. Then $K=Y/{\rm st}_P(M)$, where $Y$ is a submodule of $M$ and ${\rm st}_P(M) \subsetneq Y$. Since ${\rm st}_P(M) \subsetneq Y$, $S{\rm Hom}_R(P,Y)\neq 0$. So there exists $0 \neq \alpha \in S{\rm Hom}_R(P,Y)$. Let $\pi: Y \rightarrow Y/{\rm st}_P(M)$ be the canonical projection. It is easy to see that $0 \neq \alpha \pi \in S{\rm Hom}_R(P,K)$. This yields ${\rm st}_P({\rm sTr}(P,M/{\rm st}_P(M)))=0$.  Therefore the functor $S{\rm Hom}_R(P,-): \mathscr{C}_P \rightarrow S{\rm Mod}$ is dense. Let $0 \neq h:M \rightarrow N$ in $\mathscr{C}_P$. Since $M = {\rm sTr}(P,M)$, there exists a non-zero morphism $g \in S{\rm Hom}_R(P,M)$ such that ${\rm Im}g \nsubseteq {\rm Ker}h$. Thus the functor $S{\rm Hom}_R(P,-)$ is faithful. Now we show that the functor $S{\rm Hom}_R(P,-)$ is full. Let $\varphi:S{\rm Hom}_R(P,M) \rightarrow S{\rm Hom}_R(P,N)$ be an $S$-module homomorphism and $M, N \in \mathscr{C}_P$. Since $M \in \mathscr{C}_P$, ${\rm sTr}(P,M)=M$ and ${\rm st}_P(M)=0$. Since ${\rm sTr}(P,M)=M$ by Lemma \ref{a7}(a), we can see that there is an $R$-module epimorphism $g:\bigoplus_{\alpha \in S{\rm Hom}_R(P,M)} Pe_{\alpha} \rightarrow M$, where $\lbrace e_{\alpha}~|~ \alpha \in S{\rm Hom}_R(P,M)\rbrace$ is a set of idempotents of $S$ with $e_{\alpha}\alpha=\alpha$. Also ${\rm st}_P(M)=0$ yields ${\rm st}_P(E(M))=0$. By applying Corollary \ref{a50}, we conclude that $S{\rm Hom}_R(P,E(M))$ is an injective module in $S$Mod. By the similar argument there is an $R$-module epimorphism $h:\bigoplus_{\beta \in S{\rm Hom}_R(P,N)} Pe_{\beta} \rightarrow N$, where $\lbrace e_{\beta}~|~ \beta \in S{\rm Hom}_R(P,N)\rbrace$ is a set of idempotents of $S$ with $e_{\beta}\beta=\beta$. Moreover ${\rm st}_P(E(N))=0$ and $S{\rm Hom}_R(P,E(N))$ is an injective module in $S$Mod. Let $\varepsilon: M \rightarrow E(M)$ and $\ell: N \rightarrow E(N)$ be the canonical injections. Since $S{\rm Hom}_R(P,E(N))$ is injective in $S$Mod, there is an $S$-module homomorphism $\varphi'':S{\rm Hom}_R(P,E(M)) \rightarrow S{\rm Hom}_R(P,E(N))$ such that the following diagram is commutative
 \begin{displaymath}
\xymatrixcolsep{5pc}\xymatrix{
0 \ar[r] & S{\rm Hom}_R(P,M) \ar[r]^{S{\rm Hom}_R(P,\varepsilon)} \ar[d]_{\varphi} &
S{\rm Hom}_R(P,E(M)) \ar[d]^{\varphi''} \\
0 \ar[r] & S{\rm Hom}_R(P,N)   \ar[r]^{S{\rm Hom}_R(P,\ell)} & S{\rm Hom}_R(P,E(N)).}
\end{displaymath}
 By Lemma \ref{a6}, there exists an $R$-module homomorphism $v:E(M) \rightarrow E(N)$ such that $\varphi''=S{\rm Hom}_R(P,v)$ because ${\rm st}_P(E(N))=0$. Consider the following diagram
\begin{displaymath}
\xymatrixcolsep{5pc}\xymatrix{
S{\rm Hom}_R(P,\bigoplus_{\alpha \in S{\rm Hom}_R(P,M)} Pe_{\alpha}) \ar[r]^<<<<<<<{S{\rm Hom}_R(P,g)}  &
S{\rm Hom}_R(P,M) \ar[d]^{\varphi} \\
S{\rm Hom}_R(P,\bigoplus_{\beta \in S{\rm Hom}_R(P,N)} Pe_{\beta})   \ar[r]_<<<<<<<{S{\rm Hom}_R(P,h)} & S{\rm Hom}_R(P,N).}
\end{displaymath}
Since $S{\rm Hom}_R(P,\bigoplus_{\alpha \in S{\rm Hom}_R(P,M)} Pe_{\alpha})$ is a projective unitary left $S$-modules, by Lemma \ref{a4},  there is an $S$-module homomorphism $$\varphi':S{\rm Hom}_R(P,\bigoplus_{\alpha \in S{\rm Hom}_R(P,M)} Pe_{\alpha}) \rightarrow S{\rm Hom}_R(P,\bigoplus_{\beta \in S{\rm Hom}_R(P,N)} Pe_{\beta}),$$ such that the following diagram is commutative
\begin{displaymath}
\xymatrixcolsep{5pc}\xymatrix{
S{\rm Hom}_R(P,\bigoplus_{\alpha \in S{\rm Hom}_R(P,M)} Pe_{\alpha}) \ar[r]^<<<<<<<{S{\rm Hom}_R(P,g)} \ar[d]_{\varphi'} &
S{\rm Hom}_R(P,M) \ar[d]^{\varphi} \\
S{\rm Hom}_R(P,\bigoplus_{\beta \in S{\rm Hom}_R(P,N)} Pe_{\beta})   \ar[r]_<<<<<<<{S{\rm Hom}_R(P,h)} & S{\rm Hom}_R(P,N).}
\end{displaymath}
There exists a morphism $u \in {\rm Hom}_R(\bigoplus_{\alpha \in S{\rm Hom}_R(P,M)} Pe_{\alpha},\bigoplus_{\beta \in S{\rm Hom}_R(P,N)} Pe_{\beta})$ such that $\varphi'=S{\rm Hom}_R(P,u)$. Therefore $S{\rm Hom}_R(P, g\varepsilon v)=S{\rm Hom}_R(P,uh\ell)$. So we have $g\varepsilon v=uh\ell$. Thus we can observe that $v|_M: M \rightarrow N$. Set $t:=v|_M$. Now we claim that $S{\rm Hom}_R(P,t)=\varphi$.  For each $\alpha \in S{\rm Hom}_R(P,M)$ we have $(\alpha t)S{\rm Hom}_R(P,\ell)= \alpha t \ell = \alpha \varepsilon v= (\alpha)S{\rm Hom}_R(P,\varepsilon)\varphi''=(\alpha)\varphi S{\rm Hom}_R(P,\ell)$.  It follows that $\alpha t=(\alpha)\varphi$. So $S{\rm Hom}_R(P,t)=\varphi$. Therefore the functor $S{\rm Hom}_R(P,-):\mathscr{C}_P \rightarrow S{\rm Mod}$ is an equivalence of categories.
\end{proof}

Let $R$ be a ring and $\mathcal{U}$ be a non-empty set of unitary left $R$-modules. A unitary left $R$-module $M$ is said to be {\it generated by $\mathcal{U}$} if for every pair of distinct morphisms $f,g:M \rightarrow B$ in $R$Mod, there exists a morphism $h: U \rightarrow M$ with $U \in \mathcal{U}$ and $hf \neq hg$. Also, $\mathcal{U}$ is called a {\it generating set} for $R$Mod if every unitary left $R$-module generated by $\mathcal{U}$.  A unitary left $R$-module $U$ is called {\it generator} in $R$Mod if the set $\mathcal{U}=\lbrace U \rbrace$ is a generating set for $R$Mod (see \cite[Definition 13.1]{wi}).\\

Let $R$ and $S$ be two rings and $P$ be a unitary $R$-$S$-bimodule. We define the mapping $\mu: S \rightarrow {\rm End}_R(P)$ via $(s)\mu=\mu_s:x  \mapsto xs$ for each $s \in S$. It is easy to see that $\mu$ is a ring homomorphism. Moreover we define the mapping $\lambda: R \rightarrow {\rm End}_S(P)$ via $(r)\lambda=\lambda_r: x \mapsto rx$ for each $r \in R$. It is not difficult to show that $\lambda$ is a ring homomorphism. The bimodule $P$ is called {\it balanced} if $\lambda$ and $\mu$ are monomorphisms, $({\rm Im}\mu)({\rm End}_R(P))={\rm Im}\mu$ and $({\rm End}_S(P))({\rm Im}\lambda)={\rm Im}\lambda$, where $({\rm Im}\mu)({\rm End}_R(P))=\lbrace \sum_{i=1}^n\mu_{s_i}f_i~|~ n\in {\Bbb{N}}, ~s_i \in S ~{\rm and}~ f_i \in {\rm End}_R(P) \rbrace$ and $({\rm End}_S(P))({\rm Im}\lambda)=\lbrace \sum_{i=1}^m g_i \circ \lambda_{r_i}~|~ m\in {\Bbb{N}},~ r_i \in R ~{\rm and}~ g_i \in {\rm End}_S(P)\rbrace$ (see \cite{am}). The rings $R$ and $S$ are said to be {\it Morita equivalent} if there exists an additive covariant equivalence between the category of unitary left $R$-modules and the category of unitary left $S$-modules (see \cite{a}).\\

\begin{Cor}\label{p22}
Let $R$ and $S$ be two rings. Then the following conditions are equivalent.
\begin{itemize}
\item[$(a)$] The rings $R$ and $S$ are Morita equivalent.
\item[$(b)$] There exists a balanced unitary $R$-$S$-bimodule $P$ such that $_RP$ and $P_S$ are locally projective generators in $R{\rm Mod}$ and ${\rm Mod}S$, respectively.
\item[$(c)$] There exists a unitary $R$-$S$-bimodule $P$ such that:
\begin{itemize}
\item[$(i)$] $_RP$ is a locally projective generator in $R{\rm Mod}$;
\item[$(ii)$] The mapping $\mu: S \rightarrow {\rm End}_R(P)$ is a ring monomorphism and $({\rm Im}\mu)({\rm End}_R(P))={\rm Im}\mu$;
\item[$(iii)$] $Pf$ is a finitely generated left $R$-module for each idempotent $f \in S$.
\end{itemize}
\end{itemize}
\end{Cor}
\begin{proof}
$(a) \Rightarrow (b)$. It follows from \cite[Theorem 2.4]{am}.\\
$(b) \Rightarrow (c)$. It follows from the proof of \cite[Theorem 2.1]{am}.\\
$(c) \Rightarrow (a)$. It follows from Proposition \ref{a8} (see also \cite[Lemma 1.10 and Theorem 2.4]{am}).
\end{proof}

Finally we give a version of the Morita theorem over rings with local units in terms of their the full subcategory of finitely generated projective unitary modules.

\begin{The}\label{d}
Let $R$ and $S$ be two rings. Then the following conditions are equivalent.
\begin{itemize}
\item[$(a)$] The rings $R$ and $S$ are Morita equivalent.
\item[$(b)$] There exists an additive covariant equivalence ${\rm Proj}R\rightarrow {\rm Proj}S$ which preserves and reflects finitely generated projective unitary modules.
\item[$(c)$] There exists an additive covariant equivalence ${\rm proj}R \rightarrow {\rm proj}S$.
\end{itemize}
\end{The}
\begin{proof}
$(a)\Rightarrow(b)$. Let $R$ and $S$ be two rings which are Morita equivalent. Then there is an additive covariant equivalence $H: R{\rm Mod} \rightarrow S{\rm Mod}$ with the inverse equivalence $G: S{\rm Mod} \rightarrow R{\rm Mod}$. Then by \cite[Theorem 2.4]{am}, $H \simeq S{\rm Hom}_R(P,-)$ and $G \simeq P \otimes_S-$ for some unitary $R$-$S$-bimodule $P$ such that it is a locally projective generator left $R$-module and $Pf$ is a finitely generated left $R$-module for each idempotent $f \in S$. It is easy to see that the functor $S{\rm Hom}_R(P,-): {\rm Proj}R\rightarrow {\rm Proj}S$ is an additive equivalence with the inverse equivalence $P\otimes_S-: {\rm Proj}S\rightarrow {\rm Proj}R$. Now we show that it preserves finitely generated unitary modules. Let $X$ be a finitely generated unitary left $R$-module.  Then there exists an $S$-module epimorphism $ \bigoplus_{i \in I} Sf_i \rightarrow S{\rm Hom}_R(P,X)$, where $\lbrace f_i~|~ i\in I\rbrace$ is a set of idempotents of $S$. Since the functor $S{\rm Hom}_R(P,-): R{\rm Mod} \rightarrow S{\rm Mod}$ is dense and preserves direct sums,  we have an $S$-module epimorphism $\gamma: S{\rm Hom}_R(P,\bigoplus_{i \in I} Y_i) \rightarrow S{\rm Hom}_R(P,X)$ for some unitary left $R$-modules $Y_i$. Since the functor $S{\rm Hom}_R(P,-): R{\rm Mod} \rightarrow S{\rm Mod}$ is full, there exists an $R$-module epimorphism $h: \bigoplus_{i \in I} Y_i \rightarrow X$ such that $\gamma= S{\rm Hom}_R(P,h)$. Hence there exists an $R$-module epimorphism $h': \bigoplus_{i \in J} Y_i \rightarrow X$, where $J$ is a finite subset of $I$. It follows that $S{\rm Hom}_R(P,h'): S{\rm Hom}_R(P,\bigoplus_{i \in J} Y_i) \rightarrow S{\rm Hom}_R(P,X)$ is an $S$-module epimorphism. So there exists an $S$-module epimorphism $\bigoplus_{i \in J} Sf_i \rightarrow S{\rm Hom}_R(P,X)$. Therefore $S{\rm Hom}_R(P,X)$ is finitely generated. By the similar argument we can see that it reflects finitely generated unitary left $S$-modules. \\
$(b) \Rightarrow (c)$ is clear.\\
$(c) \Rightarrow (a)$.  Let $H: {\rm proj}S \rightarrow {\rm proj}R$ be an additive covariant equivalence of categories. We know that $\lbrace Sf, \alpha_{ff'}, \rho_{f'f}~|~f,f' \in E' \rbrace$ is a split direct system of finitely generated projective unitary left $S$-modules, where $E'$ is the set of all idempotents of $S$. So $\lbrace H(Sf), H(\alpha_{ff'}), H(\rho_{f'f})~|~f,f' \in E' \rbrace$ is a split direct system of finitely generated projective unitary left $R$-modules. Let $P$ be the direct limit of this direct system. Then $P$ is a locally projective left $R$-module. First we show that $P$ is a generator in $R$Mod. Let $X$ be a unitary left $R$-module. Then there exists an $R$-module epimorphism $\bigoplus_{i \in I}Re_i \rightarrow X$, where $\lbrace e_i~|~i \in I\rbrace$ is a set of idempotents of $R$. Also for each $i \in I$, $Re_i \cong H(Q_i)$ as $R$-modules for some finitely generated projective unitary left $S$-module $Q_i$. Since each $Q_i$ is finitely generated, there exists an $S$-module epimorphism $\bigoplus_{j \in I_i}Sf_{ij} \rightarrow Q_i$ where $I_i$ is a finite set and each $f_{ij}$ is an idempotent of $S$. Hence there exists an $R$-module epimorphism $H(\bigoplus_{j \in I_i}Sf_{ij}) \rightarrow H(Q_i)$ because $Q_i$ is a projective unitary left $S$-module. Since $H$ is additive, there exists an $R$-module epimorphism $\bigoplus_{j \in I_i}H(Sf_{ij}) \rightarrow H(Q_i)$.  Therefore there exists an $R$-module epimorphism $\bigoplus_{i \in I}\bigoplus_{j \in I_i}H(Sf_{ij}) \rightarrow X$. By Lemma \ref{a1} we have an $R$-module epimorphism $P \rightarrow H(Sf_{ij})$ for each $i \in I$ and $j \in I_i$. We conclude that there exists an $R$-module epimorphism $\bigoplus_{J}P \rightarrow X$ where $J$ is a set. Consequently $P$ is a generator in $R$Mod.  Now we show that $S \cong \underrightarrow{\lim}_{E'}{\rm End}_R(H(Sf))$ as rings. For each $f, f' \in E'$, we define $\Omega_{ff'}: {\rm End}_R(H(Sf)) \rightarrow {\rm End}_R(H(Sf'))$ via $g \longmapsto H(\rho_{f'f})gH(\alpha_{ff'})$. It is easy to see that $({\rm End}_R(H(Sf)), \Omega_{ff'})_{f,f' \in E'}$ is a direct system of rings and also ${\rm End}_R(H(Sf)) \cong {\rm End}_S(Sf) \cong fSf$ as rings for each $f \in E'$. We call this isomorphism $\gamma_f: fSf  \rightarrow {\rm End}_R(H(Sf))$. On the other hand, $S$ is a direct limit of $(fSf, t_{ff'})_{f,f' \in E'}$, where $t_{ff'}:fSf \rightarrow f'Sf'$ is the canonical injections. It is not difficult to show that the following diagram is commutative
\begin{displaymath}
\xymatrix{
fSf  \ar[d]_{t_{ff'}} \ar[r]^<<<<{\gamma_f}&
{\rm End}_R(H(Sf)  \ar[d]^{\Omega_{ff'}} \\
f'Sf' \ar[r]^<<<<{\gamma_{f'}}& {\rm End}_R(H(Sf')).}
\end{displaymath}
Therefore $S \cong \underrightarrow{\lim}_{E'}{\rm End}_R(H(Sf))$. Consequently by \cite[Theorem 2.5]{am} and Lemma \ref{a1}, the ring $R$ is Morita equivalent to the ring $S$.
\end{proof}

\section{Azumaya-Morita duality for rings with local units}

In this section, in order to have a duality between the category of finitely generated unitary left $R$-modules and the category of finitely generated unitary right $S$-modules we show that it is sufficient that $R$ be a left locally finite ring and $S$ be a right locally finite ring (see Theorem \ref{rdu3}). We conclude this section by giving a version of the Azumaya-Morita duality theorem for rings with local units in terms of their full subcategory of finitely generated projective unitary modules and full subcategory of finitely generated injective unitary modules (see Corollary \ref{rdu4}).\\

Consider the following properties for a unitary $R$-$S$-bimodule $U$:
\begin{itemize}
\item[(I)] For all idempotents $f \in S$ and $e \in R$, $Uf$ and $eU$ are finitely generated injective unitary left $R$-modules and right $S$-modules, respectively.
\item[(II)]  $\lbrace Uf~|~f^2=f \in S \rbrace$ and $\lbrace eU~|~e^2=e \in R \rbrace$ are cogenerating sets for $R$Mod and Mod$S$, respectively.
\item[(III)] $U$ is a balanced $R$-$S$-bimodule.
\end{itemize}

Let $R$ and $S$ be two rings with identity. We recall from \cite{and} that a balanced $R$-$S$-bimodule $U'$ is called Morita duality in case $_RU'$ and $U'_S$ are injective cogenerators in $R$-Mod and Mod-$S$, respectively. Every $R$-$S$-bimodule which has the properties $(I)$-$(III)$ is Morita duality. But by using \cite[Propositions 29.8 and 47.8]{wi}, the converse holds if  either $R$ is a left artinian ring or $S$ is a right artinian ring. Now we define Morita duality bimodule for rings with local units.

\begin{Def}\label{Def1}
A unitary $R$-$S$-bimodule $U$ is {\rm Morita duality} if it has the properties  $(I)$-$(III)$.
\end{Def}

Let $\mathscr{C}$ and $\mathscr{C}'$ be two additive categories. An additive contravariant functor $F: \mathscr{C} \rightarrow \mathscr{C}'$ is called {\it duality} if it is an equivalence of categories. This means that there exists an additive contravariant functor $G: \mathscr{C}' \rightarrow \mathscr{C}$ with natural isomorphisms $GF \simeq {\rm id}_{\mathscr{C}}$ and $FG \simeq {\rm id}_{\mathscr{C}'}$ (see \cite{wi}). A ring $R$ is called {\it left {\rm (resp.,} right{\rm )} locally finite} if every finitely generated unitary left (resp., right) $R$-module has finite length (see \cite{wis}). \\

The following theorem, which is the main theorem of this section, gives a generalization of the Azumaya-Morita duality theorem for rings with local units.

\begin{The}\label{rdu3}
Let $R$ and $S$ be two rings. Then the following conditions are equivalent.
\begin{itemize}
\item[$(a)$] There exists a duality between the category of finitely generated unitary left $R$-modules and the category of finitely generated unitary right $S$-modules.
\item[$(b)$] $R$ is a left locally finite ring and there exists a Morita duality $R$-$S$-bimodule.
\item[$(c)$] $S$ is a right locally finite ring and there exists a Morita duality $R$-$S$-bimodule.
\end{itemize}
In this case, if $U$ is a Morita duality $R$-$S$-bimodule, then the duality  functors are isomorphic to ${\rm Hom}_R(-,U)S: R{\rm mod} \rightarrow {\rm mod}S$ ~~ and ~~ $R{\rm Hom}_S(-,U): {\rm mod}S \rightarrow R{\rm mod}$.
\end{The}

A ring $R$ is called {\it self-duality} if there is a duality $\mathscr{D}: R{\rm mod} \rightarrow {\rm mod}R$. As a consequence of Theorem \ref{rdu3}, we give a characterization of Morita self-duality for rings with local units. Let $\Lambda$ and $\Delta$ be two rings with identity. Then by \cite[Theorem 24.8]{and}, there exists a duality between the category of finitely generated left $\Lambda$-modules and the category of finitely generated right $\Delta$-modules if and only if $\Lambda$ is a left artinian ring and there exists a Morita duality $\Lambda$-$\Delta$-bimodule if and only if $\Delta$ is a right artinian ring and there exists a Morita duality $\Lambda$-$\Delta$-bimodule. As a consequence we can see that $\Lambda$ is a Morita self-duality if and only if $\Lambda$ is an artinian ring and there exists a Morita duality $\Lambda$-$\Lambda$-bimodule.  P. N. \'{A}nh and C. Menini in \cite{ame} modified the concept of linear compactness and gave a characterization of Morita duality over rings with local units which is analogous to the classical case of rings with identity (see \cite[Theorem 1]{ame}). Let $R$ and $S$ be two rings with local units and $\mathscr{C}$ and $\mathscr{D}$ be full subcategories of $R$Mod and Mod$S$, respectively, which are both closed under submodules and factor modules and contain all finitely generated modules. P. N. \'{A}nh and C. Menini gave necessary and sufficient conditions to insure that the category $\mathscr{C}$ is dual to the category $\mathscr{D}$. We know that the category of the finitely generated modules is not necessarily closed under the submodules. In the above theorem we give necessary and sufficient conditions to insure that the category of finitely generated unitary left $R$-modules is dual to the category of finitely generated unitary right $S$-modules. As a consequence we have the following corollaries.

\begin{Cor}{\rm ( \cite[Azumaya-Morita duality theorem]{and})}
Let $R$ and $S$ be two rings with identity. Then the following conditions are equivalent.
\begin{itemize}
\item[$(a)$] There exists a duality between the category of finitely generated left $R$-modules and the category of finitely generated right $S$-modules.
\item[$(b)$] $R$ is a left artinian ring and there exists a Morita duality $R$-$S$-bimodule.
\item[$(c)$] $S$ is a right artinian ring and there exists a Morita duality $R$-$S$-bimodule.
\end{itemize}
In this case, if $U$ is a Morita duality $R$-$S$-bimodule, then the duality functors are isomorphic to
\begin{center}
${\rm Hom}_R(-,U): R${\rm -mod} $\rightarrow$ {\rm mod-}$S$ ~~and~~ ${\rm Hom}_S(-,U): ${\rm mod-}$S \rightarrow R${\rm -mod}.
\end{center}
\end{Cor}

A ring $R$ is called {\it left Morita} if there exists a ring $R'$ with local units and a duality $R{\rm mod} \rightarrow {\rm mod}R'$.

\begin{Cor}\label{rdu4}
Let $R$ and $S$ be two rings. Then the following conditions are equivalent.
\begin{itemize}
\item[$(a)$] There exists a duality between the category of finitely generated unitary left $R$-modules and the category of finitely generated unitary right $S$-modules.
\item[$(b)$] $R$ is a left Morita ring and there exists a duality ${\rm inj}R \rightarrow {\rm proj}S^{\rm op}$.
\end{itemize}
\end{Cor}
\begin{proof}
$(a) \Rightarrow (b)$. It follows from Theorem \ref{rdu3}.\\
$(b) \Rightarrow (a)$. It follows from Theorem \ref{d}.
\end{proof}

We need some preparation before proving Theorem \ref{rdu3}.\\

Let $R$ and $S$ be two rings and $U$ be a unitary $R$-$S$-bimodule. We recall from \cite{ame} that there is a natural transformation ${\rm id}_{R{\rm Mod}} \rightarrow R{\rm Hom}_S({\rm Hom}_R(-,U)S,U)$ via $R$-module homomorphisms
\begin{center}
$\varphi_X: X \rightarrow R{\rm Hom}_S({\rm Hom}_R(X,U)S,U)$ via $x \mapsto \left[\beta \mapsto (x)\beta\right]$.
\end{center}
Also there is a natural transformation ${\rm id}_{{\rm Mod}S} \rightarrow {\rm Hom}_R(R{\rm Hom}_S(M,U),U)S$ via $S$-module homomorphisms
\begin{center}
$\psi_X: X \rightarrow {\rm Hom}_R(R{\rm Hom}_S(X,U),U)S$ via $x \mapsto \left[\alpha \mapsto \alpha(x) \right]$.
\end{center}
A unitary left $R$-module (resp., right $S$-module) $X$ is called {\it $U$-reflexive} if $\varphi_X$ (resp., $\psi_X$) is an isomorphism (see \cite{ame}).\\

A left $R$-module $M$ is called {\it finitely cogenerated} if for each family $\lbrace K_i ~|~ i\in I \rbrace$ of submodules of $M$ with $\bigcap_{i\in I}K_i=0$, there exists a finite subset $I'$ of $I$ such that $\bigcap_{i\in I'}K_i=0$ (see \cite{wi}).

\begin{Pro}\label{ref3}
Let $R$ and $S$ be two rings and $U$ be a Morita duality $R$-$S$-bimodule. Then
each finitely generated unitary left $R$-module (resp., right $S$-module) is $U$-reflexive. Also if $R$ (resp., $S$) is a left (resp., right) locally finite ring, then \begin{itemize}
\item[$(a)$] For each finitely generated unitary left $R$-module $X$, ${\rm Hom}_R(X,U)S$ is a finitely generated unitary right $S$-module.
\item[$(b)$] For each finitely generated unitary right $S$-module $Y$, $R{\rm Hom}_S(Y,U)$ is a finitely generated unitary left $R$-module.
\end{itemize}
\end{Pro}
\begin{proof}
Let $X$ be a finitely generated unitary left $R$-module. Then there exists an $R$-module epimorphism $\bigoplus_{l=1}^nRe_l \rightarrow X$, where $n \in {\Bbb{N}}$ and each $e_l$ is an idempotent of $R$. Since by \cite[Proposition 20.15]{and}, each $Re_l$ is $U$-reflexive, $\bigoplus_{l=1}^nRe_l$ is $U$-reflexive. Since $Uf$ and $eU$ are injective unitary left $R$-modules and right $S$-module for all idempotents $f \in S$ and $e \in R$, respectively, the functors ${\rm Hom}_R(-,U)S$ and  $R{\rm Hom}_S(-,U)$ are exact. From $\lbrace Uf~|~ f^2=f \in S \rbrace$ and $\lbrace eU~|~ e^2=e \in R \rbrace$ are cogenerating sets for $R{\rm Mod}$ and ${\rm Mod}S$, respectively we get that  every factor of $U$-reflexive modules is $U$-reflexive. This implies that $X$ is $U$-reflexive. \\
$(a)$. Assume that $R$ is a left locally finite ring.   We know that there exists an $S$-module epimorphism $g: \bigoplus_{l \in A}f_lS \rightarrow {\rm Hom}_R(X,U)S$, where $A$ is a set and each $f_l$ is an idempotent of $S$. Then $R{\rm Hom}_S(g,U): R{\rm Hom}_S({\rm Hom}_R(X,U)S,U) \rightarrow R{\rm Hom}_S(\bigoplus_{l \in A}f_lS,U)$ is an $R$-module monomorphism. Since $X$ is $U$-reflexive and $R{\rm Hom}_S(\bigoplus_{l \in A}f_lS,U) \cong R\prod_{l \in A}R{\rm Hom}_S(f_lS, U)$ as $R$-modules, there exists an $R$-module monomorphism $\gamma: X \rightarrow R\prod_{l \in A}R{\rm Hom}_S(f_lS, U)$. On the other hand, since $R$ is left locally finite, $X$ has finite length and hence by \cite[Proposition 31.1]{wi}, it is finitely cogenerated. Therefore there exists a finite subset $B$ of $A$ such that $\gamma \pi$ is a monomorphism, where $\pi: R\prod_{l \in A}R{\rm Hom}_S(f_lS, U) \rightarrow \bigoplus_{l \in B}R{\rm Hom}_S(f_lS, U)$ is the canonical projection. Since $\bigoplus_{l \in B}R{\rm Hom}_S(f_lS, U) \cong R{\rm Hom}_S(\bigoplus_{l \in B}f_lS, U),$ as $R$-modules, there exists an $R$-module monomorphism $X \rightarrow R{\rm Hom}_S(\bigoplus_{l \in B}f_lS, U)$. Since the functor ${\rm Hom}_R(-,U)S$ is exact, we have an $S$-module epimorphism ${\rm Hom}_R(R{\rm Hom}_S(\bigoplus_{l \in B}f_lS, U),U)S \rightarrow {\rm Hom}_R(X,U)S$. Since $\bigoplus_{l \in B}f_iS$ is $U$-reflexive, there exists an $S$-module epimorphism  $\bigoplus_{l \in B}f_lS \rightarrow {\rm Hom}_R(X,U)S$. Therefore ${\rm Hom}_R(X,U)S$ is a finitely generated right $S$-module. Similarly, we can see that the statement $(b)$ holds when $S$ is a right locally finite ring. \\
$(b)$. Assume that $R$ is a left locally finite ring and $Y$ is a finitely generated unitary right $S$-module. Then there exists an $S$-module epimorphism $\bigoplus_{l=1}^mf_lS \rightarrow Y$, where $m \in {\Bbb{N}}$ and each $f_l$ is an idempotent of $S$.  Hence there exists an $R$-module monomorphism $R{\rm Hom}_S(Y,U) \rightarrow \bigoplus_{l=1}^mR{\rm Hom}_S(f_lS,U)$. Since $R{\rm Hom}_S(f_lS,U)\cong Uf_l$ as $R$-modules, there exists an $R$-module monomorphism $R{\rm Hom}_S(Y,U) \rightarrow \bigoplus_{l=1}^mUf_l$. From each $Uf_l$ is finitely generated and $R$ is left locally finite we can see that each  $Uf_l$ has finite length and so $R{\rm Hom}_S(Y,U)$ is finitely generated. Similarly, we can see that the statement $(a)$ holds when $S$ is a right locally finite ring.
\end{proof}

\begin{Lem}\label{4.3}
Let $R$ and $S$ be two rings and $U$ be a unitary $R$-$S$-bimodule.
\begin{itemize}
\item[$(a)$] If $R{\rm Hom}_S(-,U): {\rm mod}S \rightarrow R{\rm mod}$ is dense and $Uf$ is an injective module in $R{\rm Mod}$ for each idempotent $f \in S$, then $\lbrace Uf~|~f^2=f \in S \rbrace$ is a cogenerating set for $R{\rm Mod}$.
\item[$(b)$] If ${\rm Hom}_R(-,U)S: R{\rm mod} \rightarrow {\rm mod}S$ is dense and $eU$ is an injective module in ${\rm Mod}S$ for each idempotent $e \in R$, then $\lbrace eU~|~e^2=e \in R \rbrace$ is a cogenerating set for ${\rm Mod}S$.
\end{itemize}
\end{Lem}
\begin{proof}
$(a).$ Assume that $R{\rm Hom}_S(-,U): {\rm mod}S \rightarrow R{\rm mod}$ is dense and $Uf$ is an injective module in $R{\rm Mod}$ for each idempotent $f \in S$. We show that $\lbrace Uf~|~f^2=f \in S \rbrace$ is a cogenerating set for $R{\rm Mod}$. Let $g: X \rightarrow M$ be a non-zero morphism in $R$Mod. We show that there exists an $R$-module homomorphism $\gamma: M \rightarrow Uf$ such that $g\gamma \neq 0$, where $f^2=f \in S$. Let $x \in X$ such that $(x)g\neq 0$. Since there exists an idempotent $e^2=e \in R$ such that $ex=x$, $Rx \neq 0$. Let $\varepsilon: Rx \rightarrow X$ be the canonical injection. Then $\varepsilon g \neq 0$. Since $Rx$ is finitely generated, there exists a maximal submodule $K$ of $Rx$ such that ${\rm Ker}\varepsilon g \subseteq K$. Let $\pi: Rx/{\rm Ker}\varepsilon g \rightarrow Rx/K$ be the canonical projection.  Consider the following diagram
\begin{displaymath}
\xymatrix{
0 \ar[r] & Rx/{\rm Ker}\varepsilon g  \ar[d]_{\pi} \ar[r]^<<<<{\varepsilon g}&
M \\
& Rx/K. & }
\end{displaymath}
Since $Rx/K$ is a non-zero finitely generated unitary left $R$-module, there exists a non-zero finitely generated unitary right $S$-module $Y$ such that $Rx/K \cong R{\rm Hom}_S(Y,U)$ as $R$-modules. Since $Y$ is a finitely generated unitary right $S$-module, there exists a non-zero $S$-module epimorphism $\gamma: \bigoplus_{i=1}^nf_iS \rightarrow Y$, where $n \in {\Bbb{N}}$ and each $f_i$ is an idempotent of $S$. Hence $R{\rm Hom}_S(\gamma,U): R{\rm Hom}_S(Y,U) \rightarrow R{\rm Hom}_S(\bigoplus_{i=1}^nf_iS,U)$ is a non-zero $R$-module monomorphism. It follows that there exists a non-zero $R$-module monomorphism $$R{\rm Hom}_S(Y,U) \rightarrow \bigoplus_{i=1}^nR{\rm Hom}_S(f_iS,U).$$ Since $R{\rm Hom}_S(f_iS,U) \cong Uf_i$ as $R$-modules, there exists a non-zero $R$-module homomorphism $h: Rx/K \rightarrow Uf_i$, where $f_i^2=f_i \in S$.  Since $Uf_i$ is injective in $R$Mod, there exists an $R$-module homomorphism $\eta: M \rightarrow Uf_i$ such that the following diagram is commutative
\begin{displaymath}
\xymatrix{
0 \ar[r] & Rx/{\rm Ker}\varepsilon g  \ar[d]_{\pi} \ar[r]^<<<<{\varepsilon g}&
M \ar[d]^{\eta}\\
& Rx/K \ar[r]_{h}& Uf_i,}
\end{displaymath}
 $\pi h \neq 0$ and so $\varepsilon g \eta \neq 0$. Then there exists $r \in R$ such that $(rx)g \eta \neq 0$. Therefore $\lbrace Uf~|~f^2=f \in S \rbrace$ is a cogenerating set for $R$Mod.\\
$(b)$. It is similar to the proof of the part $(a)$.
\end{proof}

\begin{Lem}\label{z1}
Let $R$ and $S$ be two rings and $U$ be a unitary $R$-$S$-bimodule. If ${\rm Hom}_R(-,U)S: R{\rm mod} \rightarrow {\rm mod}S$ is a duality with the inverse duality $R{\rm Hom}_S(-,U): {\rm mod}S \rightarrow R{\rm mod}$, then $R$ is a left locally noetherian ring and  $S$ is a right locally noetherian ring.
\end{Lem}
\begin{proof}
Let $X$ be a finitely generated unitary left $R$-module. Then there exists an $R$-module epimorphism $\bigoplus_{i=1}^nRe_i \rightarrow X$, where $n \in {\Bbb{N}}$ and each $e_i$ is an idempotent of $R$. For each idempotent $e \in R$, we show that $Re$ is a noetherian left $R$-module. Let $e$ be an idempotent of $R$ and $L$ be an $R$-submodule of $Re$. Let $\pi: Re \rightarrow Re/L$ be the canonical projection.  Then ${\rm Hom}_R(\pi,U)S: {\rm Hom}_R(Re/L,U)S \rightarrow {\rm Hom}_R(Re,U)S$ is an $S$-module monomorphism. Set $T={\rm Coker}{\rm Hom}_R(\pi,U)S$. Consider the exact sequence $$0 \rightarrow {\rm Hom}_R(Re/L,U)S \overset{{\rm Hom}_R(\pi,U)S}{\rightarrow} {\rm Hom}_R(Re,U)S \overset{g}{\rightarrow} T \rightarrow 0.$$  Therefore there exists an $R$-module isomorphism $\gamma: L \rightarrow R{\rm Hom}_S(T,U)$ such that the following diagram is commutative
\begin{displaymath}
\xymatrixcolsep{5pc}\xymatrix{
0 \ar[r] & R{\rm Hom}_S(T,U) \ar[r]^{R{\rm Hom}_S(g,U)} & W \ar[r]^{\overline{\overline{\pi}}}  &
V \ar[r]& 0 \\
0 \ar[r] & L \ar[u]^{\gamma} \ar@{^{(}->}[r] &Re  \ar[u]^{\cong} \ar[r]_{\pi} & Re/L  \ar[u]^{\cong} \ar[r] & 0,}
\end{displaymath}
where the rows are exact, $W=R{\rm Hom}_S({\rm Hom}_R(Re,U)S,U)$, $V=R{\rm Hom}_S({\rm Hom}_R(Re/L,U)S,U)$ and $\overline{\overline{\pi}}=R{\rm Hom}_S({\rm Hom}_R(\pi,U)S,U)$. It follows that $L$ is a finitely generated left $R$-module. Consequently $Re$ is a noetherian left $R$-module. Then $X$ is a noetherian left $R$-module and hence $R$ is a left locally noetherian ring. By the similar argument we can see that $S$ is a right locally noetherian ring.
\end{proof}

\begin{Pro}\label{z4}
Let $R$ be a ring and $U$ be a unitary left $R$-module. Assume that $S$ is a subring of ${\rm End}_R(U)$ such that $S=S{\rm End}_R(U)$, $U$ is a unitary right $S$-module, $Uf$ is an injective module in $R$Mod for each idempotent $f \in S$. Suppose that ${\rm Hom}_R(-,U)S: R{\rm mod} \rightarrow {\rm mod}S$ is a duality with the inverse duality $R{\rm Hom}_S(-,U): {\rm mod}S \rightarrow R{\rm mod}$. Then $R$ is a left locally finite ring.
\end{Pro}
\begin{proof}
Let $M$ be a finitely generated unitary left $R$-module and $L_1 \supseteq L_2 \supseteq \ldots $ be a descending chain of submodules of $M$. Set $K=\bigcap_iL_i$. Then $M/K$ is a finitely generated unitary left $R$-module. If every finitely generated unitary left $R$-module is finitely cogenerated, then $M/K$ is finitely cogenerated. Since $\bigcap_iL_i/K=0$, there exists $n \in {\Bbb{N}}$ such that $\bigcap_{i=1}^nL_i/K=0$ and hence $\bigcap_{i=1}^nL_i=K$. It follows that $L_n=L_{n+i}$ for each $i> 0$ and hence $M$ is an artinian left $R$-module. Therefore by Lemma \ref{z1}, $M$ is an artinian and noetherian left $R$-module. It follows that $M$ has finite length. Thus it is enough to show that every finitely generated unitary left $R$-module is finitely cogenerated. Let $N$ be a finitely generated unitary left $R$-module and $\lbrace K_i ~|~ i\in I \rbrace$ be a family of submodules of $N$ such that $\bigcap_{i\in I}K_i=0$. Since by Lemma \ref{4.3}, $\lbrace Uf~|~f^2=f \in S \rbrace$ is a cogenerating set for $R{\rm Mod}$, by Lemmas \ref{4.1} and \ref{4.2}, it is easy to see that $r_{{\rm Hom}_R(N,U)S}(\bigcap_i K_i)=\sum_i r_{{\rm Hom}_R(N,U)S}(K_i)$ and hence ${\rm Hom}_R(N,U)S=\sum_i r_{{\rm Hom}_R(N,U)S}(K_i)$. Since ${\rm Hom}_R(N,U)S$ is a finitely generated unitary right $S$-module, there exists a finite subset $I'$ of $I$ such that ${\rm Hom}_R(N,U)S=\sum_{i \in I'}r_{{\rm Hom}_R(N,U)S}(K_i)$. Then $$\bigcap_{i \in I'}K_i \subseteq \bigcap_{\beta \in \sum_{i \in I'}r_{{\rm Hom}_R(N,U)S}(K_i)}{\rm Ker}\beta = \bigcap_{\alpha \in {\rm Hom}_R(N,U)S}{\rm Ker}\alpha.$$
Assume that there exits $0 \neq x \in \bigcap_{\alpha \in {\rm Hom}_R(N,U)S}{\rm Ker}\alpha$. There exists an idempotent $e \in R$ such that $ex=x$ and hence $Rx \neq 0$. Let $\ell: Rx \rightarrow N$ be the canonical injection.  By Lemma \ref{4.3}, $\lbrace Uf~|~f^2=f \in S \rbrace$ is a cogenerating set for $R{\rm Mod}$ and hence there exists an $R$-module homomorphism $h: N \rightarrow Uf$ such that $\ell h \neq 0$, where $f$ is an idempotent of $S$. It follows that there exists $r \in R$ such that $(rx)h \neq 0$. Then there exists an $R$-module homomorphism $h: N \rightarrow Uf$ with $(x)h \neq 0$. Set $\overline{h}=h\varepsilon$, where $\varepsilon: Uf \rightarrow U$ is the canonical injection. It is easy to see that $\overline{h} \in {\rm Hom}_R(N,U)S$ and $(x)\overline{h} \neq 0$ which is a contradiction. Then there exists a finite subset $I'$ of $I$ such that $\bigcap_{i \in I'}K_i=0$. Therefore every finitely generated unitary left $R$-module is finitely cogenerated and the result follows.
\end{proof}

\begin{Rem}
{\rm Let $S$ be a ring and $U$ be a unitary right $S$-module. Assume that $R$ is a subring of ${\rm End}_S(U)$ such that $R={\rm End}_S(U)R$, $U$ is a unitary left $R$-module and $eU$ is an injective module in Mod$S$ for each idempotent $e \in R$. Suppose that ${\rm Hom}_R(-,U)S: R{\rm mod} \rightarrow {\rm mod}S$ is a duality with the inverse duality $R{\rm Hom}_S(-,U): {\rm mod}S \rightarrow R{\rm mod}$. Then by the similar argument as in the proof of Proposition \ref{z4}, we can see that $S$ is a right locally finite ring.}
\end{Rem}

\begin{Lem}\label{z2}
Let $R$ and $S$ be two rings and $D_1: R{\rm mod} \rightarrow {\rm mod}S$ be a duality with the inverse duality $D_2: {\rm mod}S \rightarrow R{\rm mod}$.  Assume that $X$ and $Y$ are modules in $R{\rm mod}$ and ${\rm mod}S$, respectively.
\begin{itemize}
\item[$(a)$] If $R$ is a left locally noetherian ring and $D_1(X)$ is a projective right $S$-module, then $X$ is injective in $R{\rm Mod}$.
\item[$(b)$] If $S$ is a right locally noetherian ring and $D_2(Y)$ is a projective left $R$-module, then $Y$ is injective in ${\rm Mod}S$.
\end{itemize}
\end{Lem}
\begin{proof}
$(a)$. Assume that $R$ is a left locally noetherian ring and $D_1(X)$ is a projective right $S$-module.  We show that $X$ is injective in $R{\rm Mod}$. Consider the following diagram
\begin{displaymath}
\xymatrix{
0 \ar[r] & L  \ar[d]_{g} \ar[r]^<<<<{\ell}&
Re \\
& X, & }
\end{displaymath}
where $e$ is an idempotent of $R$, $\ell: L \rightarrow Re$ is an $R$-module monomorphism and $g: L \rightarrow X$ is an $R$-module homomorphism. Since $R$ is a left locally noetherian, $Re$ is a noetherian left $R$-module and hence $L$ is a finitely generated left $R$-module. Applying $D_1$ we have the following diagram
\begin{displaymath}
\xymatrix{
 & D_1(X)\ar[d]^{D_1(g)} \\
D_1(Re) \ar[r]_{D_1(\ell)} & D_1(L) \ar[r] & 0,}
\end{displaymath}
where its row is exact. There exists an $S$-module homomorphism $h: D_1(X) \rightarrow D_1(Re)$ such that the following diagram is commutative
\begin{displaymath}
\xymatrix{
 & D_1(X) \ar[dl]_{h} \ar[d]^{D_1(g)} \\
D_1(Re) \ar[r]_{D_1(\ell)} & D_1(L) \ar[r] & 0.}
\end{displaymath}
Then we have the following commutative diagrams
\begin{displaymath}
\xymatrix{
& L \ar@/_/[ddl]_<<<<{g} \ar[r]^{\ell} \ar[d]_{\cong}& Re \ar[d]^{\cong}\\
 & D_2D_1(L)  \ar[d]_{D_2D_1(g)} \ar[r]^<<<<{D_2D_1(\ell)}&
D_2D_1(Re) \ar[dl]^{D_2(h)} \\
X \ar[r]^<<<{\cong} & D_2D_1(X). & }
\end{displaymath}
Therefore by \cite[Lemma 1]{H}, $X$ is injective in $R$Mod.\\
$(b)$. It is similar to the proof of the part $(a)$.
\end{proof}

Now we ready to prove Theorem \ref{rdu3}:

\begin{proof}
$(b) \Rightarrow (a)$. It follows from Proposition \ref{ref3}.\\
$(a) \Rightarrow (b)$. Let $D_1: R{\rm mod} \rightarrow {\rm mod}S$ be an additive contravariant equivalence with the inverse equivalence $D_2: {\rm mod}S \rightarrow R{\rm mod}$.   Consider the split direct systems $\lbrace Re, \alpha_{ee'}, \rho_{e'e} ~|~e,e' \in E\rbrace$ and  $\lbrace fS, \alpha_{ff'}, \rho_{f'f} ~|~f,f' \in E'\rbrace$ in $R$Mod and Mod$S$, respectively. Then
\begin{center}
$\lbrace D_1(Re), D_1(\rho_{e'e}), D_1(\alpha_{ee'})~|~e,e' \in E\rbrace$ ~~and ~~$\lbrace D_2(fS), D_2(\rho_{f'f}), D_2(\alpha_{ff'})~|~f,f' \in E'\rbrace$
\end{center}  are split direct systems of finitely generated modules in Mod$S$ and $R$Mod, respectively. Set $U_S={\underrightarrow{\lim}}_ED_1(Re)$ and ${_RV}={\underrightarrow{\lim}}_{E'}D_2(fS)$. By the similar argument in the proof \cite[Theorem 1]{ame}, $U$ is a balanced unitary $R$-$S$-bimodule and $U \cong V$ as $R$-modules. It is not difficult to show that ${\rm Hom}_S(S, D_1(-))S \simeq {\rm Hom}_R(-,U)S$. Also by \cite[Proposition 1.1]{am}, $D_1(-) \simeq {\rm Hom}_S(S,D_1(-))S$. Hence $D_1(-) \simeq {\rm Hom}_R(-,U)S$. By the similar argument we can see that $D_2(-) \simeq R{\rm Hom}_S(-,U)$.  This implies that ${\rm Hom}_R(-,U)S: R$mod $\rightarrow $mod$S$ is a duality with the inverse duality $R{\rm Hom}_S(-,U):$ mod$S \rightarrow R$mod. By Lemma \ref{z1}, $R$ is a left locally noetherian ring and also $S$ is a right locally noetherian ring. Hence by Lemma \ref{z2}, for each idempotent $e \in R$ and $f\in S$, $D_1(Re)$ and $D_2(fS)$ are injective modules in Mod$S$ and $R$Mod, respectively. Since $D_1(Re) \cong {\rm Hom}_R(Re,U)S \cong eU$ as $S$-modules and $D_2(fS) \cong S{\rm Hom}_S(fS,U) \cong Uf$ as $R$-modules, by Proposition \ref{z4}, $R$ is a left locally finite ring.  Also by Lemma \ref{4.3}, $\lbrace Uf~|~f^2=f \in S \rbrace$ and $\lbrace eU~|~e^2=e \in R \rbrace$ are cogenerating sets for $R$Mod and Mod$S$, respectively. Therefore the result follows. \\
$(a) \Leftrightarrow (c)$. By symmetry.
\end{proof}

\section*{acknowledgements}
The research of the first author was in part supported by a grant from IPM. Also, the research of the second author was in part supported by a grant from IPM (No. 1400170417).

\end{document}